\documentclass[USenglish]{article}
\usepackage{fullpage}
\usepackage{amsmath,amsfonts,amsthm,amssymb}
\usepackage{color,xcolor}
\usepackage{ifpdf}
\usepackage{psfrag}
\usepackage{graphicx,graphics}
\usepackage{hhtensor}
\usepackage{comment}
\usepackage[small]{caption}
\usepackage{subcaption}

\usepackage{xparse}
\usepackage{bigints}

\usepackage{url}
\usepackage{hyperref}
\usepackage{todonotes}
\usepackage{ulem} 
\newtheorem{theorem}{Theorem}[section]

\newtheorem{remark}[theorem]{Remark}
\newtheorem{definition}[theorem]{Definition}
\newtheorem{example}[theorem]{Example}
\newtheorem{assumption}[theorem]{Assumption}

\usepackage{cancel}
\usepackage{algorithm,algorithmic}
\usepackage{enumitem}

\newcommand{\R}{\mathbb R}

\newcommand{\PP}{\mathbb P}
\newcommand{\TT}{\mathcal{T}}
\newcommand{\FF}{\mathcal{F}}

\newcommand{\EE}{\mathcal{E}}

\newcommand{\diff}[1]{{\mathrm{d}{#1}}}

\newcommand{\ee}{e}

\newcommand{\ResKs}{{{\Phi}^K_\sigma}}

\newcommand{\tildResKs}{{{\tilde{\Phi}}^K_\sigma}}
\newcommand{\hResKs}{{{\hat{\Phi}}^K_\sigma}}
        
\newcommand{\bbf}{{\mathbf {f}}}

\newcommand{\bn}{{\mathbf {n}}}

\newcommand{\dpar}[2]{\dfrac{\partial #1}{\partial #2}}
\newcommand{\bu}{\mathbf{u}}
\newcommand{\bv}{\mathbf{v}}

\newcommand{\bbg}{\mathbf{g}}

\newcommand{\bx}{\mathbf{x}}

\renewcommand{\div}{\operatorname{div}}
\newcommand{\est}[1]{\left\langle#1\right\rangle}
\newcommand{\bU}{\mathbf{U}}

\newcommand{\bV}{\mathbf{V}}

\newcommand{\mean}[1]{\overline{#1}}
\newcommand{\bbfh}{{\mathbf {f}^h}}

\newcommand{\VV}{\mathcal{V}}

\newcommand\norm[1]{\left\lVert#1\right\rVert}

\NewDocumentCommand{\mat}{mo}{%
  \IfValueTF{#2}{%
    \textrm{#1}{#2}
  }{%
    \textrm{#1}\,
  }%
}

\usepackage{bbm}

\def\R{\mathbb{R}}

\definecolor{darkspringgreen}{rgb}{0., 0.55, 0.3}
\definecolor{dartmouthgreen}{rgb}{0.05, 0.5, 0.06}
\definecolor{etonblue}{rgb}{0.59, 0.78, 0.64}
\definecolor{airforceblue}{rgb}{0., 0.4, 0.66}
\definecolor{arylideyellow}{rgb}{0.91, 0.84, 0.42}
\definecolor{emerald}{rgb}{0.31, 0.78, 0.47}
\definecolor{uclagold}{rgb}{1.0, 0.7, 0.0}
\definecolor{cadmiumorange}{rgb}{0.93, 0.53, 0.18}

  \newcommand{\bbfnum}{\mathbf{f}^\mathrm{num}}
  \newcommand{\bbgnum}{\mathbf{g}^\mathrm{num}}

\newcommand{\CU}{\mathcal{U}}

\newsavebox{\DelimiterBox}
\newlength{\DelimiterHeight}
\newlength{\DelimiterDepth}
\newsavebox{\ArgumentBox}
\newlength{\ArgumentHeight}
\newlength{\ArgumentDepth}
\newlength{\ResizedDelimiterHeight}
\newlength{\ResizedDelimiterDepth}

\newcommand{\encloseby}[3]{%
  \savebox{\ArgumentBox}{$\displaystyle #1$}%
  \settoheight{\ArgumentHeight}{\usebox{\ArgumentBox}}%
  \settodepth{\ArgumentDepth}{\usebox{\ArgumentBox}}%
  \savebox{\DelimiterBox}{#2}%
  \settoheight{\DelimiterHeight}{\usebox{\DelimiterBox}}%
  \settodepth{\DelimiterDepth}{\usebox{\DelimiterBox}}%
  \setlength{\ResizedDelimiterHeight}{%
    \maxof{1.2\ArgumentHeight}{\DelimiterHeight}%
  }
  \setlength{\ResizedDelimiterDepth}{%
    \maxof{1.2\ArgumentDepth}{\DelimiterDepth}%
  }
  \raisebox{-\ResizedDelimiterDepth}{%
    \resizebox{\width}{\ResizedDelimiterHeight+\ResizedDelimiterDepth}{%
      \raisebox{\DelimiterDepth}{#2}%
    }%
  }
  #1
  \raisebox{-\ResizedDelimiterDepth}{%
    \resizebox{\width}{\ResizedDelimiterHeight+\ResizedDelimiterDepth}{%
      \raisebox{\DelimiterDepth}{#3}%
    }%
  }
}

\ifluatex
  
  \newcommand{\jump}[1]{\encloseby{#1}{$[\mkern-4mu[$}{$]\mkern-4mu]$}}
\else
  
  \newcommand{\jump}[1]{\encloseby{#1}{$[\mkern-3mu[$}{$]\mkern-3mu]$}}
\fi

\newcommand{\bm}{\mathbf{m}}

\newcommand{\remi}[1]{\textcolor{black}{#1}}
\newcommand{\myrho}{\lambda}

\newcommand{\rev}[1]{\textcolor{black}{#1}}

\hypersetup{
pdftitle={}
pdfauthor={P. \"Offner},
pdfpagemode=UseOutlines,
linkbordercolor=0 0 0,
linkcolor=red,
citecolor=blue,
colorlinks=true,
bookmarks = true
}
\begin{document}
\title{On the convergence of residual distribution schemes for the compressible Euler equations via dissipative weak solutions}

\author{R\'emi Abgrall\thanks{Institute of Mathematics, Universit\"at  Z\"urich, Z\"urich, Switzerland}, 
 M\'aria Luk\'acova-Medvid'ov\'a\thanks{Institute of Mathematics,
Johannes-Gutenberg University Mainz, Germany},
and Philipp \"Offner\thanks{ Corresponding author: poeffner@uni-mainz.de or mail@philippoeffner.de, Institute of Mathematics,
Johannes-Gutenberg University Mainz, Germany}
}

\date{November 20, 2022}
\maketitle

\begin{abstract}
In this work, we prove the  convergence of residual distribution schemes to dissipative weak solutions of  the Euler equations. We need to guarantee that  the residual distribution schemes are fulfilling  the underlying structure preserving properties such as 
positivity of density and internal energy. Consequently, the residual distribution schemes lead to a consistent and stable approximation of the Euler equations. 
Our result can be seen as a generalization of the Lax-Richtmyer equivalence theorem to nonlinear problems that consistency plus stability is equivalent to convergence. 
\end{abstract}

\section{Introduction}

Hyperbolic conservation laws play a fundamental role within mathematical models for various physical processes, including fluid mechanics, electromagnetism and wave phenomena and  the Euler equations in gas dynamics 
are  one of the most -or even the most- investigated system in such context. 
However, due to the recent results \cite{deLellis2010admissibility, chiodaroli2015global,feireisl2020oscillatory}
of non-uniqueness of weak entropy solutions of the Euler equations, more and more attention has been given to an alternative solution concept. Measure-valued solutions (MVS) already introduced by DiPerna \cite{diperna1985compensated} has been taken up again and  further extended to dissipative weak solutions (DW). 
In a series of papers \cite{carrillo2017weak, feireisl2019uniqueness,breit2020dissipative, ghoshal2021uniqueness, feireisl2021note}, existence and weak-strong uniqueness results of the  Euler equations (barotropic, complete) has been demonstrated whereas 
in \cite{feireisl2019convergence} the convergence of some numerical schemes inside this framework has been firstly studied and demonstrated. To this goal stability and consistency of a numerical scheme is needed. 
Due to the weak-strong uniqueness principle, it can be proven  that the numerical solutions converge strongly to the strong solution on its lifespan. In \rev{ \cite{lukavcova2021convergence, feireisl2021numerics}} these results were extended to several finite volume  (FV)  methods and quite recently, 
in \cite{lukacova2022convergence}, a convergence analysis via DW solutions of a particular discontinuous Galerkin scheme has been performed. In the current work, we will extend those investigations and prove the convergence of residual distribution (RD) methods to the Euler equations via DW solutions. 
RD, also known under the name fluctuation splitting, is a uniforming framework for several high-order finite-element (FE) methods including continuous and discontinuous Galerkin, flux reconstruction, etc.. Today, RD is interpreted in a FE framework \cite{abgrall2017high}, but historically, it has been seen in a finite volume context.  The first basic idea of RD has been already  described by Roe in his seminal work \cite{roe1981approximate} where the author 
initially suggested to see the integral of the divergence of the flux of a hyperbolic conservation as a measurement of the error, i.e. a fluctuation that could be possibly evolved in such way that its decomposition in signals allows to evolve the approximation towards the sought solution. In forthcoming works \cite{roe1982fluctuations, roe1986characteristic} the RD idea has been further  extended whereas the first really pure RD scheme was properly proposed in \cite{ni1981multiple}. As mentioned before, nowadays RD is interpreted in a FE setting. \rev{First, mainly for steady-state problems \cite{deconinck2004residual}, the approach was extended to unsteady problems in several contexts, cf. \cite{abgrall2017high, ricchiuto2010explicit}. We follow this modern FE interpretation described in \cite{abgrall2017high} and demonstrate convergence of consistent structure-preserving RD schemes via dissipative weak solutions. }  \\
The rest of the paper is organized as follows: In Section \ref{se:Foundation} we introduce the definition of dissipative weak solutions for the Euler system of gas dynamics. Next, in Section \ref{sec_RD} we describe the RD schemes, repeat some basic properties and explain how we can ensure that our RD schemes 
fulfill discretely the underlying physical laws and are structure preserving. In Section \ref{se:Consistency} we further demonstrate that our RD schemes yields a consistent approximation of the Euler equations. Using this property, we can finally ensure the convergence results presented in Section  \ref{se:convergence}. We verify our theoretical results by numerical experiments. Conclusions in Section \ref{sec_conclusion} finishes the main part of this paper. Finally, in Appendix \ref{appendix}, we demonstrate that the LxF-RD scheme is positivity preserving using either an explicit or implicit time-integration method.

\section{Dissipative Weak Solutions for the Complete Euler System} \label{se:Foundation}
 
We focus  on  two-dimensional Euler equations of gas dynamics for simplicity. An extension to three dimensional problems can be done in an analogous way. The Euler equations build a hyperbolic conservation law for  conserved variables density $\rho$, momentum $\bm=\rho \bu$ and total energy 
$E=\frac{1}{2} \rho |\bu|^2 +\rho e$.   Here $e$ is the internal energy and $\bu:=(u_1, u_2)^T$ the velocity field. 
Written in a compact form we have 
\begin{equation}\label{eq_Euler_conservation}
\partial_t \bU+\div \mathbf{f}=0
\end{equation}
in the space-time domain $(t, \bx) \in (0,T)\times \Omega$. Here,  $\bU=(\rho, \bm, E)^T$ denotes the conserved vector  and $\mathbf{f}_m=(\rho u_m, u_m \bm +p \mathbf{e}_m, u_m(E+p))^T, m=1,2$ are the flux functions where $\mathbf{e}_m$ represents the m-th row of the unit matrix.
The equation of state for an ideal gas 
$p=(\gamma-1)\rho e$ with $\gamma>1 $ and pressure $p$ is used. 
We consider \eqref{eq_Euler_conservation} in the bounded space domain $\Omega \in \R^2$ and equip them with periodic or no-flux boundary conditions.
The mathematical entropy for \eqref{eq_Euler_conservation} is defined by 
\begin{equation}\label{eq:entropy_convex}
\eta=- \frac{\rho  s}{\gamma-1}
\end{equation} 
with thermodynamic entropy   $s:=\log \frac{p}{\rho^{\gamma}}$.
The corresponding entropy flux $\mathbf{g}:=(g_1, g_2)$ is defined by $g_m=\eta \cdot u_m$, $m=1,2$, with the velocity vector $\bu$. 
Due to the strict convexity of the  entropy  \eqref{eq:entropy_convex} (if $\rho >0$ and $p>0$), we can work instead of $\bU$ with the entropy variable
\begin{equation}\label{eq:entropy_variables}
\mathbf{V}=\eta'(\bU)= \left(\frac{\gamma}{\gamma-1}-\frac{s}{\gamma-1}-\frac{\rho |\mathbf{\bu}|^2}{2p}, \frac{\rho u_1}{p},\frac{\rho u_2}{p},
-\frac{\rho}{p} \right)^T.
\end{equation}
Finally, we denote by  $\Psi=\rho \bu$ the entropy potential. 
Additionally to \eqref{eq_Euler_conservation}, we require the following entropy inequality
 \begin{equation}\label{iq:entropy}
 \frac{\partial}{\partial t} \eta +\div \bf{g} \leq 0.
 \end{equation}
 
In this work, we focus on the convergence properties of the general residual distribution methods to 
\textbf{dissipative weak (DW) solution} for the Euler equations. RD includes many high order FE-based schemes like continuous discontinuous Galerkin (CG/DG), flux reconstruction (FR) and streamline upwind Petrov-Galerkin (SUPG) in a common framework \cite{offner2020stability}.

We consider the Euler equations  with either periodic boundary or no-flux boundary conditions.
For the definition, we need also the following notations from \cite{feireisl2021numerics}. 
The symbol $\mathcal{M}^+(\overline{\Omega})$ denotes the set of all positive  \textbf{Radon measures} that can be identified at the space of all linear forms on $C_c(\overline{\Omega})$, especially if $\overline{\Omega}$ is compact,
i.e. $[C_c(\overline{\Omega})]^*=\mathcal{M}(\overline{\Omega})$. Finally, the symbol  $\mathcal{M}^+(\overline{\Omega}; \R^{d\times d}_{sym})$ denotes the set of  \emph{positive semi-definite matrix valued measures}, i.e.
\begin{equation*}
\mathcal{M}^+(\overline{\Omega}, \R^{d\times d}_{sym}) = \left\{ \nu \in \mathcal{M}^+(\overline{\Omega}, \R^{d\times d}_{sym}) \big|
\int_{\overline{\Omega}} \phi(\xi \otimes\xi): \diff \nu\geq0  \text{ for any } \xi \in \R^d, \phi \in C_c(\overline{\Omega}), \phi\geq 0 
 \right\}.
\end{equation*}
With these notations, we can finally give the following definition of a dissipative weak solution following \cite{feireisl2021numerics}:
\begin{definition}[Dissipative weak solution for the Euler equations]\label{def_dmv}
Let $\Omega\subset \R^2$ be a bounded domain.  Let the initial condition $[\rho_0, \bm_0, \eta_0]$ with $\rho_0>0$ and $\int_{\Omega} \frac{1}{2} 
\frac{|\bm_0|^2}{\rho_0} + e(\rho_0, \eta_0)\diff \bx <\infty$, we call $ [\rho, \bm, \eta]$ a \textbf{dissipative weak solution}  of the complete Euler system
with periodic conditions or no-flux boundary conditions  if the following holds:
\begin{itemize}
\item  $\rho  \in C_{weak}([0,T];L^{\gamma}(\Omega) )$, 
$\bm  \in C_{weak}([0,T];L^{\frac{2\gamma}{\gamma+1}}(\Omega;\R^2) )$, $\eta \in  L^{\infty}(0, T; L^{\gamma}(\Omega)) \cap BV_{weak}([0,T];L^{\gamma}(\Omega)).$ 
\item  There exists a measure $\mathfrak{E} \in  L^{\infty}(0,T;\mathcal{M}^+(\overline{\Omega}))$ (energy defect),  such that the \textbf{energy inequality} 
\begin{equation*}
 \int_{\Omega} \left[   \frac{1}{2} \frac{|\bm|^2}{\rho} + \rho e(\rho,\eta )\right]  ( \tau,\cdot) \diff \bx + \int_{\Omega}   \diff{\mathfrak{E}}(\tau)
 \leq  \int_{\Omega} \left[ \frac{1}{2} \frac{|\bm_0|^2}{\rho_0} +\rho_0 e(\rho_0,\eta_0) \right] \diff \bx
\end{equation*} 
is fulfilled for a.a.  $0\leq \tau\leq T$.
\item The weak formulation of the  \textbf{equation of continuity}
\begin{equation*}
\left[  \int_{\Omega} \rho \varphi \diff \bx \right]_{t=0}^{t=\tau}= \int_{0}^\tau \int_{\Omega} 
\left[ \rho \partial_t \varphi + \bm \cdot \nabla_{\bx} \varphi  \right] \diff \bx \diff t
\end{equation*} 
is satisfied for any $0 \leq \tau \leq T$  and any $\varphi \in C^\infty([0,T]\times \overline{\Omega})$. 
\item The integral identity derived from the \textbf{momentum equation}
 \begin{align*}
 \left[  \int_{\Omega} \bm \cdot \mathbf{\varphi} \diff \bx\right]_{t=0}^{t=\tau}
 = &\int_{0}^\tau \int_{\Omega} 
 \left[ \bm \cdot \partial_t \mathbf{\varphi }+ 
 1_{\rho>0} \frac{ \bm \otimes  \bm }{\rho}   : \nabla_{\bx} \mathbf{\varphi} +1_{\rho>0} p(\rho, \eta)  \div_{\bx} \mathbf{\varphi} \right]
 \diff \bx \diff t 
   + \int_{0}^\tau \int_{\Omega} \nabla_{\bx} \mathbf{\varphi}: \diff{\mathfrak{R}}
 \end{align*}
holds for any $0 \leq \tau\leq T$ and any test function $  \mathbf{\varphi} \in C^\infty([0,T]\times \overline{\Omega};\R^d)$ and in case of no-flux boundary conditions 
additionally $\varphi \cdot \bn|_{\partial \Omega}=0$ there.  $
 \mathfrak{R}\in L^{\infty} \left(0,T; \mathcal{M}\left(\overline{\Omega}, \R^{d\times d}_{sym}\right) \right) 
$ is the so-called  Reynolds defect. 
 \item The  weak formulation of the  \textbf{entropy  inequality} 
\begin{equation*}
\begin{aligned}
\left[\int_{\Omega} \eta \varphi  \diff \bx \right]_{t=\tau_1-}^{t=\tau_2+} 
\leq&  \int_{\tau_1}^{\tau_2} \int_{\Omega} \left[ 
\eta \partial_t \varphi +\est{\nu; 1_{\tilde{\rho}>0} \left( \tilde{\eta} \frac{\tilde{\bm}}{\tilde{\rho}} \right) } \cdot \nabla_\bx \varphi  \right] \diff \bx \diff t\\
\eta(0^-, \cdot) =& \eta_0
\end{aligned}
\end{equation*}
holds for any $0\leq \tau_1 \leq \tau_2 <T$, 
any $\varphi \in C_c^\infty((0,T)\times \Omega), \varphi\geq 0$, where $\{ \nu_{t,\bx}\}_{(t, \bx) \in (0,T) \times \Omega}$ is a 
parametrized (Young) measure
\begin{equation}\label{eq:Young}
\begin{aligned}
\nu \in L^{\infty}((0,T)\times \Omega;\mathcal{P}(\mathcal{F})), \mathcal{F}=\left\{ \tilde{\rho} \in \R, \tilde{\bm}\in \R^d, \tilde{\eta}\in \R  \right\};\\
\est{\nu, \tilde{\rho}} =\rho, \est{\nu, \tilde{\bm}} =\bm, \est{\nu, \tilde{\eta}} =\eta,\\
\nu_{t, \bx}\left\{ \tilde{\rho} \geq 0, (1-\gamma)\tilde{\eta} \geq \underline{s} \tilde{\rho} \right\}=1 \text{ for a. a. } (t,\bx) \in (0,T)\times \Omega; 
\end{aligned}
\end{equation} 
\item There exists constants $0\leq c_1 \leq c_2$  such that the  \textbf{defect compatibility condition} on
$
c_1 \mathfrak{E} \leq \operatorname{tr} [\mathfrak{R}] \leq c_2 \mathfrak{E}
$
holds.
\end{itemize}
\end{definition}
The advantage of dissipative weak solutions is that the solution concept is compatible with the classical solution concept. 
As it is shown in \cite{feireisl2021numerics}  if $[\rho, \bm, \eta]$ belongs to 
\begin{equation}\label{eq_regularity}
\rho \in C^1([0,T]\times \overline{\Omega}), \; \inf_{(0,T)\times \Omega} \rho>0, \; \mathbf{u} \in C^1([0,T]\times \overline{\Omega}; \R^d), \; \eta \in C^1([0,T]\times \overline{\Omega})
\end{equation}
then $[\rho, \bm, \eta]$ is a classical solution of the complete Euler system. Therefore, if a classical solution exists the DW solutions coincide with the classical one.

\section{Residual Distribution Schemes}\label{sec_RD}
In the following part, we describe shortly the numerical method under consideration the general residual distribution scheme and introduce the used notation.
We repeat their basic properties and demonstrate how entropy dissipative RD schemes can be constructed following 
\cite{abgrall2018general, abgrall2019reinterpretation}. 
Finally, we also introduce the multi-dimensional optimal order detection (MOOD) algorithm from \cite{clain2011high}, \rev{applied in \cite{zbMATH07300580, zbMATH07517733, zbMATH07568430, zbMATH07524776} in different frameworks and how it is used in our RD context} following \cite{bacigaluppi2019posteriori}. We use this approach to ensure the positivity of pressure and density inside the calculation as well as a discrete maximum principle. This is essential since the physical constraints have to be fulfilled.
Now, we focus on the general residual distribution approach. 
In the following, we give some short remarks on the behaviour of RD schemes  concerning 
the properties needed for our convergence analysis.

\subsection*{Geometrical Notations}
Before we start, we fix the main notation used in the manuscript. 
The computation domain $\Omega$ is covered by a grid $\TT_h$ (triangles, quads, general polygons). We denote by $\EE_h$ the set of internal edges /faces of $\TT_h$ and $\FF_h$ the set of boundary faces. $K$ denotes the generic mesh element, while we use $\ee$ for a face/edge $e \in \EE_h \subset \FF_h$. We assume that the mesh is shape regular, $h_K$ represents the diameter of the element $K$, $|K|$ its area and $h=\max_K h_K$ Similarly, we have $h_{\ee}$ for faces/edges. We use a classical FE approximation and follow Ciarlet's definition 
\cite{ciarlet2002finite}. We have a set of degrees of freedom (DOFs) $\sum_K$ of linear forms acting on the set $\PP^p$ of polynomials of degree $p$ such that the linear mapping 
$$
q\in \PP^p \to (\sigma_1(q), \cdots, \sigma_{|\sum_K|}(q))
$$
is one-to-one. The space $\PP^p$ is the set of polynomials of degree less or equal to $p$. It is spanned by the basis function $\{ \phi_\sigma \}_{\sigma \in \sum_K}$ 
defined by $\forall \sigma, \sigma', \sigma (\phi_{\sigma'}) = \delta_{\sigma}^{\sigma'}$. Here, we denote by $\sigma$ a generic DOF and $\delta$ is the classical Kronecker symbol. Elements of such representation are either Lagrange polynomials or B\'ezier/Bernstein polynomials where the DOFs are associated to points in $K$. It is important that for any $K$ the following properties holds:
$
\forall \bx \in K, \quad \sum_{\sigma \in K} \phi_{\sigma}=1.
$
We define by 
\begin{equation}
\VV^h = \bigoplus_K \left\{ \bv \in L^2(K), \bv_K \in \PP^p \right\}
\end{equation}
and since we are working with continuous or discontinuous FE schemes, we are searching our solutions for the Euler equation \eqref{eq_Euler_conservation} 
\begin{enumerate}
\item[S1] either in $V^h=\VV^h$, 
\item [S2]or $V^h= \VV^h \cap C^0(\Omega)$. 
\end{enumerate}
In S1, we allow discontinuities across internal edges of $\TT_h$. We are in the classical DG or FR setting \cite{offner2020stability}. Here, we need no conformity requirement on the mesh which is needed in continuous case S2.
Finally, for $\ee \in \EE_h$ represent any internal edge, i.e. $\ee \subset K \cap K^+$ for two neighboring elements $K$ and $K^+$, and we define the mean value $\mean{\bv}= \frac{1}{2}( \bv|_K+ \bv|_{K^+})$ and jump $\jump{\bv}= \bv|_{K^+}- \bv_{K}$. 
Finally, we make the following assumption on the considered grid:
\begin{assumption}
\label{H0}
The mesh $\TT_h$ is conformal and shape regular. By shape regular, we mean that all elements
 are roughly the same size, more 
precisely that there exist constants 
 $C_1$ and $C_2$ such that for any element $K$:
$$ \; \; C_1 \leq \sup_{K \in \mathcal{K}_h} \dfrac{h^2}{|K|} \leq C_2.$$
\end{assumption}
Assuming that the mesh is conformal, is mostly for simplicity in the case S1, while it is mandatory in the case S2.
We say that two elements are neighbours if they have a common edge.

The symbol $\oint$ will be used for a surface or boundary integral computed with a quadrature formula.

\subsection*{Residual Distribution Schemes}

RD is a uniform framework for several high-order FE based methods \cite{offner2020stability}. For simplification, 
we explain first RD for a steady state problem 
\begin{equation}\label{steadyversion}
\div \bbf (\bU) = 0.
\end{equation}
Due the FE element ansatz 
 an approximation of a solution of \eqref{steadyversion} is expressed by a linear combination of basis functions of $V^h$:
\begin{equation}
\bU(\bx) \approx \bU^{h} = \sum_{\sigma \in \Omega_h} \bU_\sigma \phi_\sigma(x), \; x \in \Omega, \;\phi_\sigma \in V^h . \label{fem}
\end{equation}
In terms of the Euler equations, the representation is done for each component. Finally, we have to calculate the coefficients at any DOF, the approach works in three steps:
\begin{enumerate}
\item Define $\Phi^K(\bU) :=\int_{\partial K} \bbfnum (\bU^{h,K}, \bU^{h,K^+}, \bn) \text{d} \gamma$,
where $\bbfnum$ is any consistant numerical flux. 
\item Define the local/element residual $\Phi_\sigma^K$ as the contribution of a DOF $\sigma$ to the total residual of the element $K$.
It is 
\begin{equation}\label{eq:residual_basic}
  \sum\limits_{\sigma \in K} \ResKs(\bU^h) =\Phi^K(\bU^h) 
\end{equation} 
\item Finally, all local residuals belonging to one DOF $\sigma$ are collected and summed up. This gives the equation for that DOF $\bU_\sigma$, i.e.,
\begin{equation}\label{eq:RDequationDof}
\sum_{K|\sigma \in K} \Phi_\sigma^K= 0, \quad \forall \sigma \in K.
\end{equation}
\end{enumerate}
The advantage of RD is its generality. No further constraints are formulated besides the fact that 
 the element residual have to fulfill the following conservation relation:
\begin{equation}\label{eq:residual_conservation}
  \sum\limits_{\sigma \in K} \ResKs(\bU^h) = \int_{\partial K} \bbfnum (\bU^{h,K}, \bU^{h,K^+}, \bn) \text{d} \gamma,
\end{equation} 
where $\bbfnum$ is any consistant numerical flux. In the case of S2, this reduces to the flux and above relation simplify to 
\begin{equation}\label{eq:residual_conservation_CFE}
  \sum\limits_{\sigma \in K} \ResKs(\bU^h) = \int_{\partial K}  \bbf(\bU^h)\cdot \bn \;  \text{d} \bx.
\end{equation} 

In both cases, the integrals are evaluated by quadrature. We assume throughout the paper that the quadrature points are defined on the edges/faces of the elements $K$. This implies that for any edge/face $e$ that is shared by $K$ and $K'$
$\oint_e\bbfnum (\bU^{h,K}, \bU^{h,K^+}, \bn) \text{d} \gamma$ seen from $K$ is $-\oint_e\bbfnum (\bU^{h,K}, \bU^{h,K^+}, \bn) \text{d} \gamma$ seen from $K'$ because we have the \emph{same} quadrature points and the outward normals are opposite.
This will ensure local conservation.

The final discretization of \eqref{steadyversion} reads: for any $\sigma$, 
\begin{equation}
\sum_{K\in \Omega| \sigma \in K} \ResKs(\bU^h)+ \textbf{Boundary terms}=0.
\end{equation}
We omit the discussion about boundary terms for simplicity as it is mainly a technical, but fundamental, detail, cf. \cite{offner2020stability}.
The 
 choice of $ \Phi_\sigma^K$ together with the underlying approximation space $V^h$ fully determines the scheme and the framework we are working in. As mentioned before, RD is a unifying framework. We give now a couple of examples that are included in our further analysis.

 \begin{example}\label{Various_RD}$ $
 \begin{itemize}
 \item 
 A pure continuous Galerkin discretization can be written in residual form as
\begin{equation}\label{Galerkin}
 \Phi_\sigma^K(\bU^h) = \int_K \phi_\sigma \nabla \cdot \bbf (\bU^h) \diff \bx = \int_{\partial K} \phi_\sigma \bbf (\bU^h) \cdot \bn \diff \gamma - \int_K \nabla \phi_\sigma \cdot \bbf (\bU^h) \diff \bx.
\end{equation}
\item For a Galerkin discretization with jump stabilization, the residual are defined by:
\begin{equation} \label{eq_Galerkin_jump}
 \Phi_\sigma^K(\bU^h) = \int_{\partial K} \phi_\sigma \bbf (\bU^h) \cdot \bn \diff \gamma - \int_K \nabla \phi_\sigma \cdot \bbf (\bU^h) \diff \bx 
 + \sum_{e \in \EE} \lambda_{\ee} h^2_\ee \int_\ee \jump{ \nabla \bU^h } \cdot \jump{\nabla \phi_\sigma } \diff \gamma,
\end{equation}
where $\lambda$ is a stabilization coefficient \cite{burman2004edge, michel2021spectral}.
 \end{itemize}
 In both expressions \eqref{Galerkin} - \eqref{eq_Galerkin_jump} above, the mesh has to be conformal. 
 \begin{itemize}
 \item The DG discretization can be written in residual form as
\begin{equation}
 \Phi_\sigma^K(\bU^h) = \int_{\partial K} \phi_\sigma \bbfnum (\bU^{h,K}, \bU^{h,K^+})\diff \gamma - \int_K \nabla \phi_\sigma \cdot \bbf (\bU^h) \diff \bx,
 \end{equation}
 \item while the FR residuals are defined by:
 \begin{equation}
 \Phi_\sigma^K(\bU^h) = \int_{\partial K} \phi_\sigma \bbfnum (\bU^{h,K}, \bU^{h,K^+})\diff \gamma - \int_K \nabla \phi_\sigma \cdot \bbf (\bU^h) \diff \bx
 - \underbrace{\int_K \nabla \phi_\sigma \cdot \mathbf{\alpha} \nabla \Psi \diff \bx}_{:= \mathcal{C}_\sigma}
 \end{equation}
 with the following constraints $ \nabla \Psi \equiv 1$ and $\sum_{\sigma \in K} \mathcal{C}_\sigma=0$, cf. \cite{huynh2007flux, trojak2022extended, trojak2022extended_2, offner2020stability, cicchino2022nonlinearly}. 
 \item Instead of working with these classical schemes we can use RD itself to formulate a new method. 
Here, the so-called nonlinear Lax-Friedrich RD scheme should be named. 
It is given by  
 $$
 \Phi_\sigma^K = \beta_{\sigma}^K  \Phi^K,
 $$
 where the coefficient $\beta$ are designed in such a way that the scheme is both monotonicity preserving 
 and formally $p+1$th order accurate if a polynomial approximation of order $p$ is applied. 
 As described in \cite{abgrall2011construction_2} this can be done following the two steps:
 \begin{enumerate}
 \item First evaluate the Rusanov/local Lax-Friedrich (LxF) residuals
 \begin{equation}\label{eq_LxF_residual}
 \Phi_\sigma^{K, LxF}= \frac{\Phi^K}{N_K} +\alpha_K( \bU_{\sigma} - \mean{\bU}^K),
 \end{equation}
 where $\alpha_K$ is the larger than the maximum on $K$ of $\norm{\nabla \bbfh}_{L^\infty}$, $N_K$ is the number of DOFs on $K$ 
 and $\mean{\bU}^K$ is the arithmetic average of $\bU_{\sigma}$ for $\sigma \in K$.
 \item Define $\bx_{\sigma}$ as the ratio of  $\Phi_\sigma^{LF,K}$ by $\Phi^K$, and
 $\beta_{\sigma} = \frac{ \max(\bx_{\sigma} ,0)}{ \sum_{\sigma' \in K} \max(\bx_{\sigma} ,0)} .
 $
 \end{enumerate}
 \end{itemize}
\end{example}
 All of the above residuals will be considered in the following text. 
 Finally for the generic residual $\Phi_\sigma^K(\bU^h)$, we define its stencil $\mathcal{S}_\sigma$, the set of degrees of freedom, that are needed to evaluate it, in other words,
$$\Phi_\sigma^K(\bU^h)=\Phi_\sigma^K\big(\bU_{\sigma'}, \sigma'\in \mathcal{S}_\sigma\big).$$
Then, we make the additional assumption on the residual:
 \begin{assumption}
\label{H2} Let $\TT_h$ be a triangulation satisfying Assumption \ref{H0}.
For any $C\in \R^+$, there exists $C'( C, \TT_h)\in \R^+$ which depends only on 
$C$ and $\TT_h$ such that for any 
$ \bU \in V^h, \text{ with } \norm{\bU}_{L^{\infty}(\R^2)}\leq C$ we have 
\begin{equation}
\label{eq:conservation}
\forall K, \forall \sigma,\quad  \norm{\Phi_{\sigma}^{K} (\bU)} \leq C'(C, \TT_h) \; h^{d-1}\;
\sum_{\sigma'\in \mathcal{S}_\sigma}\norm{\bU_{\sigma'}-\bU_{\sigma}},
\end{equation}
where $\norm{\cdot}$ denotes the Euclidian norm.
\end{assumption}
Note that $d$ is set to two in our case.  Here we introduce the general setting. We will also demonstrate the consistency estimation in Section \ref{se:Consistency} for general dimension $d$ for completeness. Finally, we use $C$ as a generic constant in the following part.

\subsection*{Entropy Correction Term.}\label{se:Correction}
It will be essential that our RD schemes fulfill the entropy inequality for the Euler equations. 
To this end, we follow the approach presented by Abgrall \cite{abgrall2018general} and add an entropy correction term to our steady-state residual. Therefore, we apply that $\eta:\R^I \to\R$ is a strict convex entropy with the corresponding entropy flux 
 $\bbg:\R^{d+2}\to \R^{d}$. The entropy variable is $\partial_\bU \eta(\bU)=\mathbf{V} \in \R^{d+2}$ such that $\langle \eta'(\bU),\bbf'(\bU)\rangle = \bbg'(\bU)$, cf. \cite{harten1983symmetric}. With $\mathbf{V}^h\in V^h$, we denote the approximated entropy variable. 
The entropy equality in the conservative case using the RD framework reads
\begin{equation}\label{eq:entropy_condition}
\sum_{\sigma\in K } <\mathbf{V}_\sigma;\tildResKs> =
\int_{\partial K} \bbg \left(\mathbf{V}^h \right)\cdot \bn \diff \gamma,
\end{equation}
where $\tildResKs$ is a modification of the previously presented residuals.
Since \eqref{eq:entropy_condition} is not fulfilled for general $\ResKs$,
the entropy correction terms $ r_\sigma^K$ is added to the residuals $\ResKs$ to guarantee \eqref{eq:entropy_condition}.
In addition, we have to select these correction terms
such that they do not violate the conservation relation. We introduce the following definition of the entropy-corrected residuals
\begin{equation}\label{eq:residual_corr}
\tildResKs= \ResKs+r_\sigma^K
\end{equation}
with the goal of fulfilling the discrete entropy condition \eqref{eq:entropy_condition}.
In \cite{abgrall2018general}, the following correction terms are presented
\begin{align}
 \label{eq:correction}
 r_\sigma^K := \alpha_K(\mathbf{V}_\sigma -\mean{\mathbf{V}} ),
 \quad \text{with }
\mean{\mathbf{V}} = \frac{1}{N_K} \sum_{\sigma \in K } \mathbf{V}_\sigma,
 \\
 \label{eq:error_abgrall}
 \alpha_K = \frac{E}{\sum\limits_{\sigma \in K } \left(\mathbf{V}_\sigma -\mean{\mathbf{V}} \right)^2},
 \quad
 E:= \int_{\partial K} \bbg \left(\mathbf{V}^h \right) \cdot\bn \diff \gamma-
 \sum_{\sigma\in K }<\mathbf{V}_\sigma;\ResKs> ,
\end{align}
where $N_K$ denotes the number of DOFs belonging to $K$.
 By adding \eqref{eq:correction} to the residual $\ResKs$,
 the resulting scheme using $\tildResKs$ is locally conservative in $\bU$ and
 entropy conservative.
 However, entropy conservation is most of the time not enough for the Euler equations of gas dynamics since the presence of  
discontinuities (i.a. shocks), the scheme should not just fulfill the equality \eqref{eq:entropy_condition} but rather an inequality 
\begin{equation}\label{eq:entropy_condition_second}
\sum_{\sigma\in K }<\mathbf{V}_\sigma;\hResKs> \geq
\int_{\partial K} \bbg \left(\mathbf{V}^h\right)\cdot \bn \diff \gamma.
\end{equation}
 To obtain a semi-discrete entropy dissipative scheme in the continuous FE case, 
 we apply the previous construction and write the new residual as 
 \begin{equation}\label{eq:residual_corr_diss}
\hResKs= \ResKs+r_\sigma^K+\Psi_\sigma^K,
\end{equation}
 where $r_\sigma^k$ are defined by \eqref{eq:correction}
and jump 
diffusion $\Psi_\sigma^K$, defined similarly to \eqref{eq_Galerkin_jump}, by 
\begin{equation}\label{eq:jumpStabilization}
\Psi_\sigma^K := \lambda h_K^2 \int_{\partial K} \jump{\nabla \phi_\sigma} \cdot \jump{\nabla \mathbf{V}^h} \diff \gamma,
\end{equation}
which ensures that 
\begin{equation}\label{eq_ensure}
 \sum_{\sigma \in K} \est{\mathbf{V}_\sigma, \Psi_\sigma^K} =  
 \lambda h_K^2 \int_{\partial K} \jump{\nabla \mathbf{V}^h}^2 \diff \gamma \geq 0
\end{equation}
for any $\lambda>0$ and so the strict inequality in \eqref{eq:entropy_condition_second}.
In the discontinuous case we use instead of the gradients directly the jumps of the quantities (a local Lax-Friedrich dissipation term). 
 For entropy dissipative RD schemes \eqref{eq:residual_corr_diss}  a weak BV estimation for the Euler equations  is proven later and used in the consistency estimation, see Section \ref{se:Consistency}.

\subsection*{Extension to Unsteady Flows.}\label{se:Correction}

RD can be seen as an arbitrarily high-order FE discretization developed for steady-state problems, 
the generalization to time-dependent problems should not be done via a classical method of lines approach since 
 that would require the inversion of the mass matrix (not desirable in continuous FE) and decrease the order. First, some correction terms \cite{ricchiuto2010explicit} in the RK setting where proposed, later 
 the deferred correction (DeC) method \cite{abgrall2017high} have been used to overcome those issues.
The RD scheme for time-dependent problems reads than in semidiscrete form 
\begin{equation}\label{semi_unsteady}
|C_{\sigma}| \partial_t \bU_\sigma(t) = -\sum_{K| \sigma \in K} \ResKs (\bU_{\sigma} (t)),
\end{equation}
where $ \ResKs$ is our element residual and $C_{\sigma}$ denotes the dual control volume associated with the DOF $\sigma$.

\subsection*{Basic Properties of RD schemes and the MOOD Extension.}
Due to the entropy correction terms from Subsection \ref{se:Correction} we can ensure  locally the entropy inequality
for all of our mentioned RD schemes from Example \ref{Various_RD} to be fulfilled. However, we need further properties for the convergence and consistency analysis, e.g. the positivity of pressure and density inside the numerical calculations. 
To this end we use as well the \textbf{a posteriori multidimensional optimal order detection} (MOOD) method proposed in \cite{clain2011high}. It was further extended and applied in the RD context in \cite{bacigaluppi2019posteriori}. The main idea of MOOD is the following, one starts with any high-order method and calculate the solution at every degree of freedom in the next time step. Here, the solutions are checked by several criteria, e.g. the positivity of density and pressure for the Euler equations. If the solution passes all criteria the solution is accepted if one (or several criteria) are not fulfilled the solution at this DOF will be rejected. The element is marked and one calculates the solution inside the element again using a scheme with more favourable properties (classically more dissipation)
and the same procedure begins. In the end, we have a so-called parachute scheme with the best properties (most dissipation) which yield the desired result. A sketch of the procedure is exemplary given in Figure \ref{fig_Mood}. The LxF scheme \eqref{eq_LxF_residual} is working in our case always as the parachute scheme, therefore it would be enough to repeat the basic properties of this scheme only. However, we also consider other RD methods. We repeat the basic properties of all considered RD schemes and name the specific properties for some of them. These properties are also related to the detection criterias. Here, we follow the approach from \cite{bacigaluppi2019posteriori} and repeat the algorithm. For details, we 
refer to \cite{bacigaluppi2019posteriori}. 
\begin{figure}[tb]
	\center
		\includegraphics[width=0.8\textwidth]{%
   		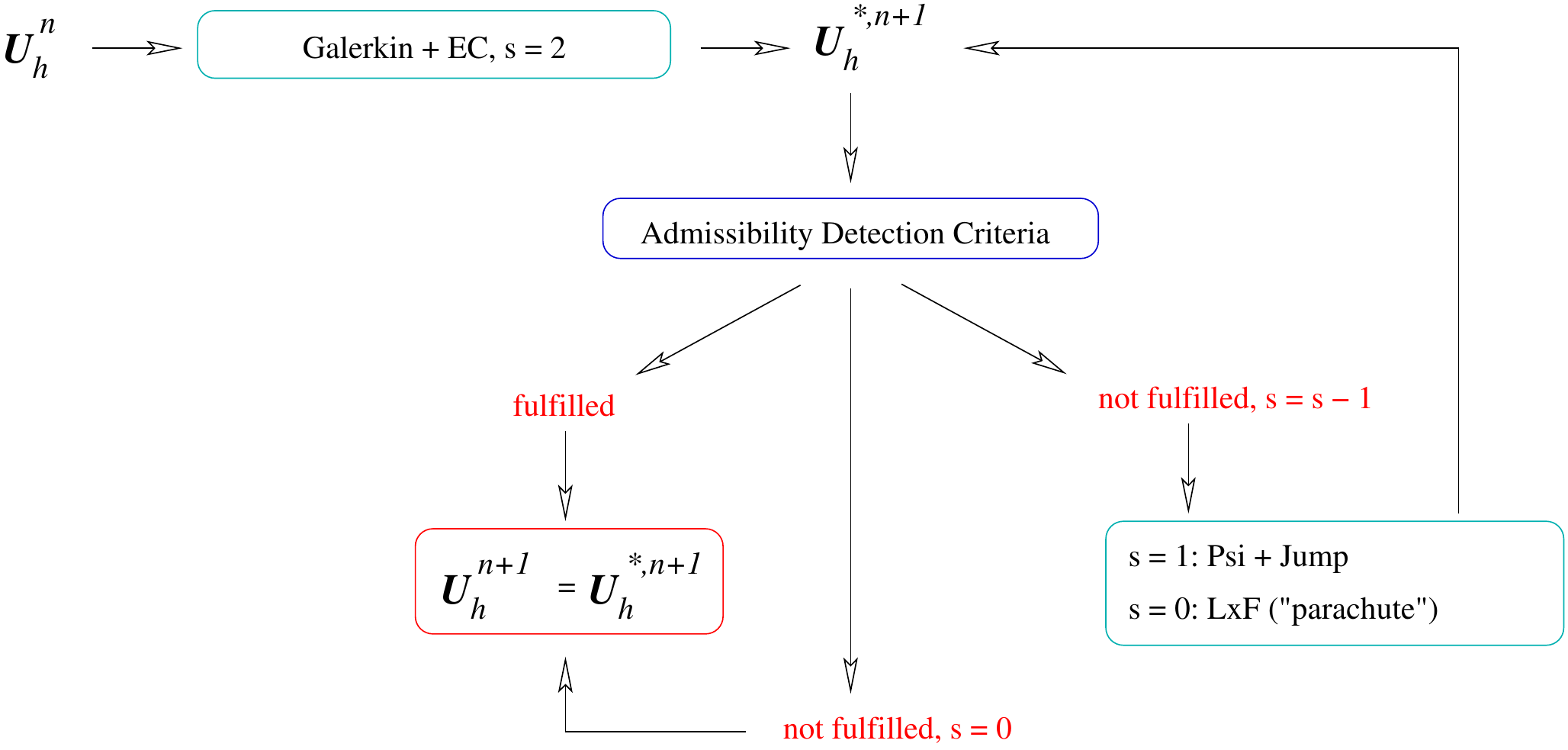} 
 	\caption{MOOD procedure with Galerkin and entropy correction, Psi+ jump stabilization and LxF-RD schemes}
 	\label{fig_Mood}
\end{figure}  
Before we start we require  the following assumption on the numerical approximation 
$\rho^h(t), \bm^h(t):=\rho^h(t) \bu^h(t), E^h(t) \in V^h$ of $\rho(t), \bm(t), E(t)$
obtained by our RD scheme \eqref{semi_unsteady}. 
 \begin{assumption}\label{assumtion_1}
 We assume that there exist two positive constants $\underline{\rho}$ and $\overline{E}$ such that
\begin{equation}\label{vacuum}
\rho^h(t) \geq\underline{\rho} >0 \qquad E^h(t) \leq \overline{E} \text{ uniformly for } h\to 0.
\end{equation}
 \end{assumption}
The physical meaning of this assumption is that no vacuum appears. 
The second assumption \eqref{vacuum} implies then that the speed $|\bu^h|$ is bounded since 
$
 |\bu^h|^2\leq \frac{2E^h}{\rho^h} \leq \frac{2\overline{E}}{\underline{\rho}}<C.
$
As it is described in  \cite{lukavcova2021convergence,feireisl2019convergence}, Assumption \ref{assumtion_1} implies that the density is also bounded from above and the energy is bounded from below. Consequently, the pressure and temperature are bounded from above and below as well. 

\subsubsection*{Conservation}
 
 The main feature of RD is the fact that it is interpreted in fluctuation splitting form which is related to the conservative form used in classical finite volume/finite difference schemes. All of the RD schemes are locally conservative as long as condition \eqref{eq:residual_conservation} holds. Since it is working locally, the conservation is also ensured even if the spatial 
 discretization scheme differs between two neighboring elements which are essential important if the MOOD procedure is applied.

Actually, Assumption \ref{H2} is  important to prove the Lax-Wendroff  theorem  \cite{abgrall2003high}.
We need this assumption as well inside our consistency analysis.
One should see this assumption as asserting continuity of the 
residual components (or signals) $\Phi_\sigma^{K}$ with respect to the 
nodal values of $\bU$; in particular, when $\bU$ is constant, 
$\Phi_\sigma^{K}=0$. Note that the proof of the Lax-Wendroff theorem when  
$\Phi_{\sigma}^{K}$ satisfies Assumption \ref{H2} is still valid if the number of arguments in $\Phi_\sigma^{K}$
is bounded independently of $h$ and the element $K$. In practice,
 this is always true if the triangulation
is uniform, since the arguments of $\Phi_\sigma^{K}$ are contained in some
neighborhood of $\sigma$ comprising a finite number of points.
Finally, we like to point out that using the entropy correction term does not affect those results due to its conservation property, i.e. $\sum\limits_{\sigma \in K} r_{\sigma}^K=0$.

\subsubsection*{High-order accuracy for smooth solutions}
There exist plenty of papers in the literature focusing on the high-order accuracy properties of RD schemes for steady and unsteady flows, cf. \cite{hubbard2011discontinuous,  abgrall2017high, abgrall2011construction_2} in case a smooth solution is approximated. Here, we want to point out that due to the analysis of 
\cite{abgrall2018general, offner2020stability} the application of entropy correction terms does not affect this property if 
sufficiently accurate  quadratures formulas are used. Roughly speaking, 
for a sufficiently smooth solution, we obtain an $h^{p+1}$ order accurate approximation both in space in time using RD together with the DeC approach \cite{abgrall2017high}. 


\subsubsection*{Positivity Preservation}
Not all of the residuals named in Example \ref{Various_RD} can ensure to keep the density and pressure (internal energy) positive for the Euler equations \eqref{eq_Euler_conservation}. However, at least in our parachute scheme, the LxF residual \eqref{eq_LxF_residual} should provide this property. In  Appendix \ref{positivity}, we demonstrate that this is true under a certain CFL condition for both the time explicit  and implicit RD schemes. 

\subsubsection*{Detection Procedure}

The detection procedure is essential and besides the positivity of density, other points are checked. We base our detections on physical/modelling and numerical considerations. The algorithm procedure is the following from \cite{bacigaluppi2019posteriori}:
\begin{itemize}
\item Physical Admissibility Detection: The density and pressure at every degree of freedom have to remain positive. 
\item Computational Admissibility Detection: The numerical solution can not be undefined (Not-A-Number or Infinity). 
\item Plateau Detection: If we are on a plateau, we make sure that we do not break that area. 
\item Numerical Admissibility Detection: The solution is tested against oscillatory behaviour. Here, we test for a relaxed discrete maximum principle. If the criteria are activated, we test for smoothness meaning that we have natural oscillations inside the calculations. 
\end{itemize}
Using the MOOD procedure, we can start with any RD scheme and fall back, in the worst case scenario, to the parachute scheme. \remi{Here, the parachute scheme is the local Lax-Friedrich (or Rusanov distributions). Its properties (positivity, entropy dissipation) are recalled in the Appendix \ref{appendix}. }Due to the finite number of cells and loops, the procedure converges.
In the following part, we use the MOOD procedure to ensure the positivity of density and pressure at every degree of freedom. 


\section{Consistency Analysis}\label{se:Consistency}

We investigate the consistency of our semi-discrete RD scheme \eqref{semi_unsteady}.
We will show that for the  numerical solution  $\bU^h=(\rho^h, \bm^h, E^h)$ calculated by semi-discrete RD  
\begin{equation}\label{equ:consistent}
 \left[\int_{\Omega} \bU^h \cdot \varphi\diff \bx \right]_{t=0}^{t=\tau}=\int_0^\tau \int_{\Omega}  \partial_t \varphi \cdot  \bU^h + \bbf(\bU^h): \nabla_\bx \varphi \diff \bx \diff t +\int_0^\tau \mathbf{e}^h(t,\varphi) \diff t
\end{equation}
holds for all $\varphi \in C^{p+1}( [0,T]\otimes \overline{\Omega}, \R^{2+d})$ where the error  $\mathbf{e}^h\to 0$ if $h\to 0$. Next, we  specify  $\mathbf{e}^h$ and demonstrate how we can ensure \eqref{equ:consistent}. 
Note that the consistency of the total energy, cf. Definition \ref{def_dmv} (energy inequality), follows from the global conservation property of the scheme. Therefore, we restrict ourself in the following investigation on the density $\rho$ and momentum $\mathbf{m}$. However, to demonstrate the consistency estimation, we need a weak BV estimate where we focus on the entropy behavior. Therefore, we will describe as well the error of the entropy inequality. We derive finally the consistency formulation for   $\CU^h=(\rho^h, \bm^h, \eta^h)$. 
First, we realize that  for all $\varphi \in C^{p+1}( [0,T]\times \overline{\Omega}, \R^{2+d})$
\begin{equation}\label{eq_consistent_2}
  \left[\int_{\Omega} \bU^h  \varphi\diff \bx \right]_{t=0}^{t=\tau}
  = \int_0^\tau \int_{\Omega} \frac{\diff{} }{\diff t} \left( \bU^h \varphi\right)\diff \bx \diff t 
  = \int_0^\tau \int_{\Omega} \bU^h \partial_t \varphi +  \varphi \partial_t \bU^h   \diff \bx\diff t
\end{equation}
and for the last term, we have to   apply the RD scheme in the semi-discrete setting \eqref{semi_unsteady}  for $\partial_t \bU^h$ and derive the error which we have to estimate.

 \subsection{Consistency Errors}

\remi{The consistency analysis is independent of the choice of the residual. The only thing that matter, is that for the whole series of schemes used in the MOOD procedure, each scheme has the same total residual on each element: this guaranties local conservation. The choice of the residual will play a role in the entropy dissipation structure, cf. Section \ref{weakBV}.}

\medskip

For any $\varphi\in C^{p+1}( [0,T]\times \overline{\Omega}, \R^{2+d})$, $\Pi_h \varphi$ will be
$$\Pi_h\varphi=\sum\limits_\sigma\varphi_\sigma \phi_\sigma$$
its approximation in $V^h$, and $t\in [t_n, t_{n+1}[$
$$\tilde{\varphi}=\sum\limits_\sigma \varphi_\sigma^n 1_{C_\sigma}.$$
 We can rewrite  as 
$$
\left[\int_{\Omega} \bU^h  \varphi\diff \bx \right]_{t=0}^{t=\tau}=\left[\int_{\Omega} \bU^h  \big (\varphi-\tilde{\varphi}\big )\diff \bx \right]_{t=0}^{t=\tau}+\left[\int_{\Omega} \bU^h  \tilde{\varphi}\diff \bx \right]_{t=0}^{t=\tau},$$
and, since $|C_\sigma|=\int_\Omega\phi_\sigma$, under $h=O(\Delta t)$, we get
$$\left \vert \left[\int_{\Omega} \bU^h  \big (\varphi-\tilde{\varphi}\big )\diff \bx \right]_{t=0}^{t=\tau}\right \vert \leq C h\max_{t\in [0,\tau]}\Vert \bU^h\Vert_{L^1}.$$
Then, we write
$$\left[\int_{\Omega} \bU^h  \tilde{\varphi}\diff \bx \right]_{t=0}^{t=\tau}= \int_\Omega \tilde{\varphi}_t\cdot \bU^h+\tilde{\varphi}\partial_t \bU^h \diff \bx .$$ 
We have
\begin{equation*}
\begin{split}
\int_\Omega\tilde{\varphi}\partial_t \bU^h \diff \bx &=\sum\limits_\sigma \vert C_\sigma \vert \tilde{\varphi}_\sigma \partial_t \big (\bU^h_\sigma \big)=
-\sum_K\sum\limits_{\sigma\in K} \tilde{\varphi}_\sigma \Phi_\sigma^K(\bU^h)\\
&=-\sum_K \sum\limits_{\sigma\in K} \tilde{\varphi}_\sigma\bigg ( -\int_K \nabla\phi_\sigma\bbf(\bU^h)\diff \bx+
\int_{\partial K} \phi_\sigma \bbfnum (\bU^{h,K}, \bU^{h,K^+})\diff \gamma\\
&\qquad +\sum_K\sum_{\sigma\in K} \tilde{\varphi}_\sigma \big ( \Phi_\sigma^K(\bU^h)-\Phi_\sigma^{K, Gal}(\bU^h)\big ),
\end{split}
\end{equation*}
where we have set
$$\Phi_\sigma^{K, Gal}(\bU^h)=-\int_K \nabla\phi_\sigma\bbf(\bU^h)\diff \bx+
\int_{\partial K} \phi_\sigma \bbfnum (\bU^{h,K}, \bU^{h,K^+})\diff \gamma.$$
{Since $\sum_{\sigma\in K}\Phi_\sigma^K(\bU^h)=\sum_{\sigma\in K}\Phi_\sigma^{K, Gal}(\bU^h)$,  and using the conditions at the boundary (periodic or no-flux boundary), we can rewrite this as }
\begin{equation*}
\begin{split}
\int_\Omega\tilde{\varphi}\partial_t \bU^h \diff \bx &=-\int_\Omega \nabla\Pi_h\varphi \cdot \bbf(\bU^h)\diff \bx+
\sum_K\sum_{\sigma, \sigma'\in K}\big ( \varphi_\sigma-\varphi_{\sigma'}\big ) \big (\Phi_\sigma^K(\bU^h)-\Phi_{\sigma}^{K, Gal}(\bU^h)\big ).
\end{split}
\end{equation*}

Collecting all the pieces together, we get \eqref{equ:consistent} with
\begin{equation}\label{eq:error_eh}
\mathbf{e}^h(t,\varphi)=\underbrace{\sum_K\sum_{\sigma, \sigma'\in K}\big ( \varphi_\sigma-\varphi_{\sigma'}\big ) \big (\Phi_\sigma(\bU^h)-\Phi_{\sigma}^{K, Gal}(\bU^h)\big )}_{(I)}+\underbrace{\bigg [\int_\Omega \bU^h\big (\varphi-\tilde{\varphi}\big ) \diff\bx \bigg ]_{t=0}^{t=\tau}}_{(II)}+\underbrace{\int_\Omega \big ( \Pi_h\varphi-\varphi\big ) : \bbf(\bU^h)\diff \bx}_{(III)}.
\end{equation}
We have already seen that for a bounded sequence $\bU^h$ the terms $(II)$ and $(III)$ will converge to $0$, so the only thing to study is the behavior of the term $(I)$. Term $(I)$ will be discussed in the following part where we demonstrate that this term tends  to zero. Therefore, we need a weak BV estimation derived from the entropy inequality.

\subsection{Weak BV estimates for entropy dissipative RD schemes}\label{weakBV}
For the consistency estimation of the entropy inequality and to estimate term $(I)$, we need additionally to demonstrate 
the weak BV estimation for our RD scheme \eqref{semi_unsteady}.  Before, we 
note that  Assumption \ref{assumtion_1}
is related to the mathematical entropy function \eqref{eq:entropy_convex}  as demonstrated in 
 \cite[Lemma 3.1 and B2]{lukavcova2021convergence}, it is equivalent to the strict convexity of the mathematical entropy function \eqref{eq:entropy_convex}, i.e.
 \begin{equation} \label{eq:hessian}
 \exists \underline{\eta}_0 >0: \frac{\diff{}^2\eta(\bU^h) }{\diff{\bU}^2} \geq \underline{\eta}_0 \mathcal{I}
 \end{equation}
 where $\mathcal{I}$ is a unity  matrix.
 
 Now, we start with our semi-discrete entropy RD scheme \eqref{eq:residual_corr_diss}. We recall its construction. Starting for a family of residuals $\Phi_\sigma^K(\bU^h)$ that satisfies the conservation relations
 $$\sum\limits_{\sigma\in K} \Phi_\sigma^K(\CU^h)=\int_{\partial K} \bbfnum(\bU^{h,K}, \bU^{h, K^+}, \bn)\diff \gamma.$$
 By introducing a correction of the form $\alpha_K \big ( \bV_\sigma-\overline{\bV}_K\big )$, we can choose $\alpha_K$ such that $\Psi_\sigma^K(\bU^h)=\Phi_\sigma^K(\bU^h)+\alpha_K \big ( \bV_\sigma-\overline{\bV}_K\big )$ satisfies in each element
 $$\sum\limits_{\sigma\in K} \Psi_\sigma^K(\bU^h)=\int_{\partial K} \bbfnum(\bU^{h,K}, \bU^{h, K^+}, \bn)\diff \gamma,$$
 and
 $$\sum\limits_{\sigma\in K} \langle\bV_\sigma, \Psi_\sigma^K(\bU^h)\rangle=\int_{\partial K} \bbgnum(\bU^{h,K}, \bU^{h, K^+}, \bn)\diff \gamma,$$
 where the numerical flux $\bbgnum$ is consistent with the entropy flux $\bbg$.
 Then we modify again the residuals and define\footnote{If we work with discontinuous FE like DG or FR, we can apply a local Lax-Friedrich diffusion term in \eqref{eq_residual} instead of the gradient jumps resulting in analogous results.}
\begin{equation}\label{eq_residual}
\Theta_\sigma^K(\bU)=\Psi_\sigma^K+\myrho h_K^{\zeta}\int_{\partial K} \jump{\nabla \phi_\sigma} \cdot \jump{\nabla \bV^h} \diff \gamma
\end{equation}
with $\zeta\geq 2$. 
 Here, $\myrho>0$ depends on the maximum of the eigenvalues of the Jacobian matrices. 
 This residual still satisfies the conservation requirement and  we have
 $$\sum\limits_{\sigma\in K}\langle \bV_\sigma, \Theta_\sigma^K(\bU^h)\rangle=\int_{\partial K} \bbgnum(\bU^{h,K}, \bU^{h, K^+}, \bn)\diff \gamma+\myrho h_K^{\zeta}\int_{\partial K} \norm{ \jump{\nabla \bV^h}}^2 \diff \gamma,$$
where $\norm{\cdot}$ denotes the Euclidian norm in the following. \remi{In the case of the parachute scheme, the entropy production is built in into the residual, and we have a similar formula with $\zeta=2$, this is why $\zeta=2$ is the factor of choice.}

 Using this, and proceeding as before, for any positive test function $\varphi\in C^{p+1}( [0,T]\times \overline{\Omega}, \R)$, we get
 \begin{equation}\label{equ:entropyconsistent}
 \left[\int_{\Omega} \eta^h
 \varphi\diff \bx \right]_{t=0}^{t=\tau}=\int_0^\tau \int_{\Omega}  \partial_t \varphi \cdot  \eta^h 
 + \bbg(\bU^h): \nabla_\bx \varphi \diff \bx \diff t +\int_0^\tau \mathbf{e_\eta}^h(t,\varphi) \diff t.
\end{equation}
Setting $\Xi^K_\sigma(\bU^h)=\langle \bV_\sigma,\Theta^K_\sigma(\bU^h)\rangle$ and $\Xi_\sigma^{K,Gal}=\langle \bV_\sigma,\Theta^{K,Gal}_\sigma(\bU^h)\rangle$ to simplify the notations, we have 
\begin{equation}\label{eq_help_estimation}
\begin{split}
\mathbf{e_\eta}^h(t,\varphi)&=\underbrace{\sum_K\sum_{\sigma, \sigma'\in K}\big ( \varphi_\sigma-\varphi_{\sigma'}\big ) \big (\Xi_\sigma(\bU^h)-\Xi_{\sigma}^{K, Gal}(\bU^h)\big )}_{(I)}+\underbrace{\bigg [\int_\Omega \bV^h\big (\varphi-\tilde{\varphi}\big ) \diff\bx \bigg ]_{t=0}^{t=\tau}}_{(II)}+\underbrace{\int_\Omega \big ( \Pi_h\varphi-\varphi\big ) : \bbg(\bU^h)\diff \bx}_{(III)} \\&+ 
\underbrace{\sum\limits_K \myrho h_K^{\zeta}\int_{\partial K}  \norm{ \jump{\nabla \bV^h}}^2 \diff \gamma}_{(IV)}.
\end{split}
\end{equation}
In particular, taking $\varphi=1$, we obtain
\begin{equation}\label{eq:weakBV}
\left[\int_{\Omega} \eta^h\diff\bx
  \right]_{t=0}^{t=\tau}+\sum\limits_{e\in \mathcal{E}} \myrho h_K^{\zeta}\int_e\norm{ \jump{ \nabla \bV^h} }^2 \diff \gamma=0
\end{equation}

Let us define
$$\norm{\Vert \bU^h\Vert}^2 :=\sum\limits_{e\in \mathcal{E}}  h_K^{\zeta}d\int_e \norm{\jump{\nabla \bV^h}}^2 \diff \gamma. $$
This is a norm: If $\norm{\Vert \bU^h\Vert}=0$, this means that across any edge, $\jump{\nabla \bV^h}=0$, i.e. $\bV^h$ is a given polynomial  over $\R^{d}$, and because it is compactly supported, $\bU^h=0$. We will have under the assumptions of boundedness from above and below (for the density) that
$$\Vert\Vert \bU^h\Vert\Vert\leq C(\bU^h(0))\leq C' \int_\Omega \vert \bV(\bx, 0)\vert \diff \bx.$$

We have, for any $C$ and $C'$, 
$$\int_e \norm{ \jump{ \bV^h}}^2 \diff \gamma=\int_e \Vert \nabla (\bV^h-C)_{|K}-\nabla(\bV^h-C')_{|K'}\Vert^2.$$
Since the  number of degrees of freedom in $K$ and $K'$ is bounded, and because the mesh is regular, there exists $\alpha$ and $\beta$ independent of $K$ and $K'$ such that
\begin{equation}\label{weakBV:norm}\alpha h_K^{d+\zeta-3} \sum\limits_{\sigma, \sigma'\in K\bU \cup K'}\norm{\bV_\sigma-\bV_{\sigma'}}^2 \leq 
h_K^{\zeta} \int_e \norm{ \jump{ \nabla \bV^h}}^2 \diff \gamma \leq \beta h_K^{d+\zeta-3} \sum\limits_{\sigma, \sigma'\in K \cup K'}\Vert \bV_\sigma-\bV_{\sigma'}\Vert ^2.
\end{equation}
Note that $d+\zeta-3\geq d-1$.
In the end, the norm $\norm{\Vert~.~\Vert}$ and
$$\sum_{K} \sum_{\sigma\in K} h_K^{\frac{d+\zeta-3}{2}}\sum_{\sigma_1, \sigma_2\in \mathcal{S}_\sigma}\Vert \bU_{\sigma_1}-\bU_{\sigma_2}\Vert $$
are equivalent, because for $\sigma\in K$, $\mathcal{S}_\sigma$ is contained in the set of DOFs in $K$, and those of the neighbouring elements to $K$. We recall that  $\myrho$ depends on the maximum of $\bU^h$ over the mesh which is bounded.
We point out that \eqref{eq:weakBV} together with \eqref{weakBV:norm}  yields the BV estimate, namely
$$
\sum_{K}\sum_{\sigma \in K}   h_K^{\frac{d+\zeta-3}{2}}\sum_{\sigma_1, \sigma_2\in \mathcal{S}_\sigma}\Vert \bU_{\sigma_1}-\bU_{\sigma_2}\Vert \leq C. 
$$

Combining \eqref{weakBV:norm} with the Assumption \ref{H2} on the residuals, the relation \eqref{eq:error_eh} and in particular its term $(I)$,  since  $\vert \varphi_\sigma-\varphi_{\sigma'}\vert \leq C \Vert \nabla \varphi\Vert_{\infty} h$, we see that
$$\Vert \mathbf{e}_\eta^h (t, \varphi)\Vert \leq C h^{d-1+1-\frac{d+\zeta-3}{2}}=Ch^{d-\frac{d+\zeta-3}{2}}\rightarrow 0$$
when $h\rightarrow 0$\remi{, as soon as $\zeta\geq 2$.}
Here $(I)$ corresponds to the first term in \eqref{eq:error_eh} as well as in \eqref{eq_help_estimation}.
 The constant $C$ only depends on the $L^\infty$ bound on the numerical solution, cf. Assumption \ref{assumtion_1}, where the  constant $C$ from Assumption \ref{H2}  depends only on the geometrical regularity of the mesh.

\bigskip

In total, we have shown the following results:

\begin{theorem}[Consistency Formulation]\label{th_consistency}
 Let $\bU^h$ be a solution of the RD scheme with the MOOD approach on the interval $[0,T]$ with the initial data $\bU^h_0$. Under our assumptions \ref{H0}, \ref{H2}, and  \ref{assumtion_1} we have the following results for 
 all $\tau \in (0,T]$: 
 \begin{itemize}
  \item for all $\varphi \in C^{p+1}([0,T] \times \overline{\Omega})$:
  \begin{equation}\label{eq:consistency_rho}
 \left[  \int_{\Omega} \rho^h \varphi \diff \bx \right]_{t=0}^{t=\tau} =\int_0^\tau \int_{\Omega} \rho_h \partial_t \varphi +\bm^h \cdot \nabla_{\bx} \varphi \diff \bx \diff t
 +\int_0^\tau e_{\rho^h} (t,\varphi) \diff t;
  \end{equation}

  \item for all $\varphi \in C^{p+1}([0,T] \times \overline{\Omega};\R^d)$:
  \begin{equation}\label{eq:consistency_m}
 \left[  \int_{\Omega} \bm^h \mathbf{\varphi} \diff \bx \right]_{t=0}^{t=\tau} =\int_0^\tau \int_{\Omega} \bm^h \partial_t \mathbf{\varphi} +\frac{\bm^h\otimes\bm^h}{\rho^h} :  \nabla_{\bx}\mathbf{\varphi}  +p^h \div_{\bx} \mathbf{\varphi} \diff \bx \diff t + \int_0^\tau e_{\bm^h} (t,\mathbf{\varphi}) \diff t;
  \end{equation}
  
  \item for all $\varphi \in C^{p+1}([0,T] \times \overline{\Omega}), \; \varphi\geq 0$:
  \begin{equation}\label{eq:consistency_eta}
 \left[  \int_{\Omega} \eta^h \varphi \diff \bx \right]_{t=0}^{t=\tau} \leq \int_0^\tau \int_{\Omega} \eta^h \partial_t \varphi + (\eta^h \bu^h) \cdot \nabla_\bx \mathbf{\varphi} \diff \bx \diff t + \int_0^\tau e_{\eta^h} (t,\varphi) \diff t;
  \end{equation}
  \item
  $ \int_{\Omega} E^h(\tau)\diff \bx =\int_{\Omega} E^h_0 \diff \bx$
  \item The errors  $e_{j^h}$, $(j=\rho, \bm, \eta)$ tend to zero under mesh refinement 
  \begin{equation}\label{eq:consistency_error}
   \norm{e_{j^h}}_{L^1(0,T)} \to 0 \text{ if } h\to 0. 
  \end{equation}
 \end{itemize}
\end{theorem}
  
  \begin{remark}[MOOD and LxF-Residual]
Due to our assumptions, we have proven that all of our considered schemes are consistent. However, it is known that some of those schemes will have stability issues and violate the positivity of pressure and density  inside our numerical simulations. To overcome this issue, we use the MOOD approach locally. As our parachute scheme, we apply the LxF residual which has the highest amount of dissipation and acts around the shock analogously to the local Lax-Friedrich schemes where convergence for dissipative weak solutions has been proven in  \cite{feireisl2019convergence} also for the fully-discrete case.
 \end{remark}

\section{Convergence to dissipative weak solutions}\label{se:convergence}

We have  demonstrated that the RD methods yield a consistent approximation for the Euler equations \eqref{eq_Euler_conservation} Theorem \ref{th_consistency}. Due to the MOOD approach
we can also ensure that the underlying physical laws, e.g. positivity of density and pressure, are fulfilled. 
At all, we demonstrated that our final implemented scheme is high-order, structure preserving and consistent. 
These are exactly the properties which are needed to prove a convergence result in the spirit of 
 \cite{feireisl2019convergence, lukavcova2021convergence, lukacova2022convergence}. 
 We start with the weak convergence theorem similar to \cite{lukacova2022convergence}.

\begin{theorem}[Weak convergence]\label{eq:theorem_weak}
Let $\CU^h=\{\rho^h, \bm^h, \eta^h \}_{h\to 0}$ be a family of numerical solutions generated by the RD schemes \eqref{semi_unsteady} with the MOOD.  Let Assumptions  \ref{H0}, \ref{assumtion_1}, and \ref{H2} hold. Then there exists a subsequence $\CU^h$ (denoted again by $\CU^h$) such that 
\begin{equation}\label{eq_bounds}
\begin{aligned}
\rho^h &\to \rho \text{ weakly-(*) in } L^{\infty}((0,T)\times \Omega)\\
\eta^h &\to \eta \text{ weakly-(*) in } L^{\infty}((0,T)\times \Omega)\\
\bm^h &\to \bm \text{ weakly-(*) in } L^{\infty}((0,T)\times \Omega; \R^2))\\
\end{aligned}
\end{equation}
as $h\to 0$ and  $(\rho, \bm, \eta)$ is a DW solution of the complete Euler system \eqref{eq_Euler_conservation}. 
\end{theorem}
\begin{proof}
The proof follows analogous  steps as presented in  \cite{feireisl2021numerics, lukacova2022convergence}.
It uses the fact that the RD schemes with the MOOD approach lead to consistent and stable approximation of the Euler equations. 
\end{proof}

In numerical simulations weak convergence is not really suitable for visualization of a DW solution. 
It is more convenient to work with the Cesaro averages, known as $\mathcal{K}$-convergence in such context.
As demonstrated in  \cite[Theorem 10.5]{feireisl2021numerics}, we obtain  
 strong convergence of the Cesaro averages  
to a DW solution as well as strong convergence of the approximate deviation of the associated Young measures.
Here, we mean with \textbf{strong convergence of Cesaro averages} $\CU^{h_n}=(\rho^{h_n},\bm^{h_n},\eta^{h_n})$ that 
\begin{equation*}
\frac{1}{N} \sum_{n=1}^N \CU^{h_n} \to \CU \text{ as } N\to \infty \text{ in } L^q ((0,T)\times \Omega, \R^4) \text{ for any } 1\leq q<\infty.
\end{equation*}

In addition,  if we have more informations  on the  regularity of the solutions, we get the strong convergence of the sequence of approximated solutions by adapting the proof    \cite[Theorem 10.6]{feireisl2021numerics} to our RD setting.
\begin{theorem}[Strong Convergence of the RD scheme]\label{th_convergence}
Let $\CU^h=\{ \rho^h, \bm^h, \eta^h \}_{h\to 0}$ be numerical solutions of  RD scheme \eqref{semi_unsteady}   with the MOOD approach.  Let the initial data
be $\rho^h_0, \bm_0^h$ and $\eta_0^h, \rho_0 \geq \underline{\rho}>0, (1-\gamma)\eta_0 \geq \underline{\rho s}$.
 Further, let Assumptions  \ref{H0}, \ref{assumtion_1}, and \ref{H2} hold. 
 Then, the following holds:
\begin{itemize}
\item \textbf{weak solution}: \\
If $\CU=[\rho, \bm, \eta]$ obtained as a weak$^*$limit of $\{\rho^h, \bm^h,\eta^h\}_{h\to0}$, is a weak entropy solution and emanating from the initial data $\CU_0$, then 
$\nu_{t,\bx}=\delta_{\CU(t,\bx)} $ for a.a. $(t,\bx) \in (0,T) \times \Omega$, and 
\begin{align*}
(\rho^{h},\bm^h, \eta^h) &\to (\rho,\bm, \eta) \text{ in } L^q((0,T)\times \Omega;\R^4),\\
E(\CU^h)=\frac12 \frac{|\bm^{h}|^2}{\rho^{h}} + \rho^{h} e(\rho^{h}, \eta^{h}) &\to \frac12 \frac{|\bm|^2}{\rho} + \rho e(\rho, \eta)  \text{ in } L^q((0,T)\times \Omega)
\end{align*}
for any $1 \leq q <\infty$
\item \textbf{strong solution:}\\
Suppose that the Euler system admits a strong solution $\CU$ in the class 
$\rho, \eta \in W^{1,\infty}((0,T) \times \Omega), \bm \in W^{1,\infty}((0,T)\times \Omega; \R^2)$,
$\rho \geq \underline{\rho}>0$ in $[0,T]\times \Omega$  emanating from the initial data $\CU_0$. 
Then, for any $1\leq  q<\infty$ and $h\to 0$ 
\begin{align*}
(\rho^{h},\bm^h, \eta^h) &\to (\rho,\bm, \eta) \text{ in } L^q((0,T)\times \Omega; \R^4),\\
E(\CU^h)&\to E(\CU)  \text{ in } L^q((0,T)\times \Omega).
\end{align*}

\item \textbf{classical solutions:}\\
Let $\Omega\in \R^d$ be a bounded Lipschitz domain and
 $\rho\in C^1([0,T] \times \overline{\Omega})$, $\rho\geq \overline{\rho}>0,\; \bm\in C^1([0,T]\times \overline{\Omega};\R^2),\; \eta\in C^1 ([0,T] \times \overline{\Omega})$. Then $\CU=(\rho, \bm, \eta)$ is a classical solution to the Euler system and 
\begin{align*}
(\rho^{h},\bm^h, \eta^h) &\to (\rho,\bm, \eta) \text{ in } L^q((0,T)\times \Omega,\R^4)\\
\end{align*}
as $h\to 0$, for any $q\geq1$. 
\end{itemize}
\end{theorem}
\begin{proof}
 Analog to \cite[Theorem 10.6]{feireisl2021numerics}.
\end{proof}

\section{Numerical Experiments}\label{sec_numerics}

In the following section, we verify  Theorem \ref{th_convergence} by a  numerical experiment. 
We consider a moving vortex. Initially, an isentropic perturbation $(\delta S=0)$ is applied to the system, such that 
\begin{equation*}
\begin{aligned}
\delta u &= - \mean{y} \frac{\beta}{2 \pi} \ee^{(1-r^2)/2},\\
\delta v &=  \mean{x} \frac{\beta}{2 \pi} \ee^{(1-r^2)/2},\\
\delta T &= -  \frac{(\gamma-1)\beta^2}{8 \gamma \pi^2} \ee^{(1-r^2)/2}.
\end{aligned}
\end{equation*}
The initial conditions are thus set to 
\begin{equation*}
\begin{aligned}
\rho &=T^{1/(\gamma-1)} = (T_\infty+ \delta T)^{1/(\gamma-1)}  = \left[1-   \frac{(\gamma-1)\beta^2}{8 \gamma \pi^2} \ee^{(1-r^2)/2} \right ]^{1/(\gamma-1)},  \\
\rho u &=   \rho( u_{\infty}+ \delta u) =\rho\left[1- \mean{y} \frac{\beta}{2 \pi} \ee^{(1-r^2)/2} \right],\\
\rho v &= \rho( v_{\infty}+ \delta v) =\rho\left[1+ \mean{x} \frac{\beta}{2 \pi} \ee^{(1-r^2)/2} \right],\\
\rho E &= \frac{\rho^\gamma}{\gamma-1} +\frac{1}{2} \rho (u^2+v^2),
\end{aligned}
\end{equation*}
where $\beta=5, r= \mean{x}^2+\mean{y}^2$ and $(\mean{x}, \mean{y}) =((x-x_0),(y-y_0))$. The computational domain is a square with length $10$ and center $(0,0)$. 
Periodic boundary conditions are considered where we identify the left and right boundaries 
and the lower and upper boundaries with each other. 
\begin{figure}[tb]
	\centering 
	\begin{subfigure}[b]{0.32\textwidth}
		\includegraphics[width=\textwidth]{%
      		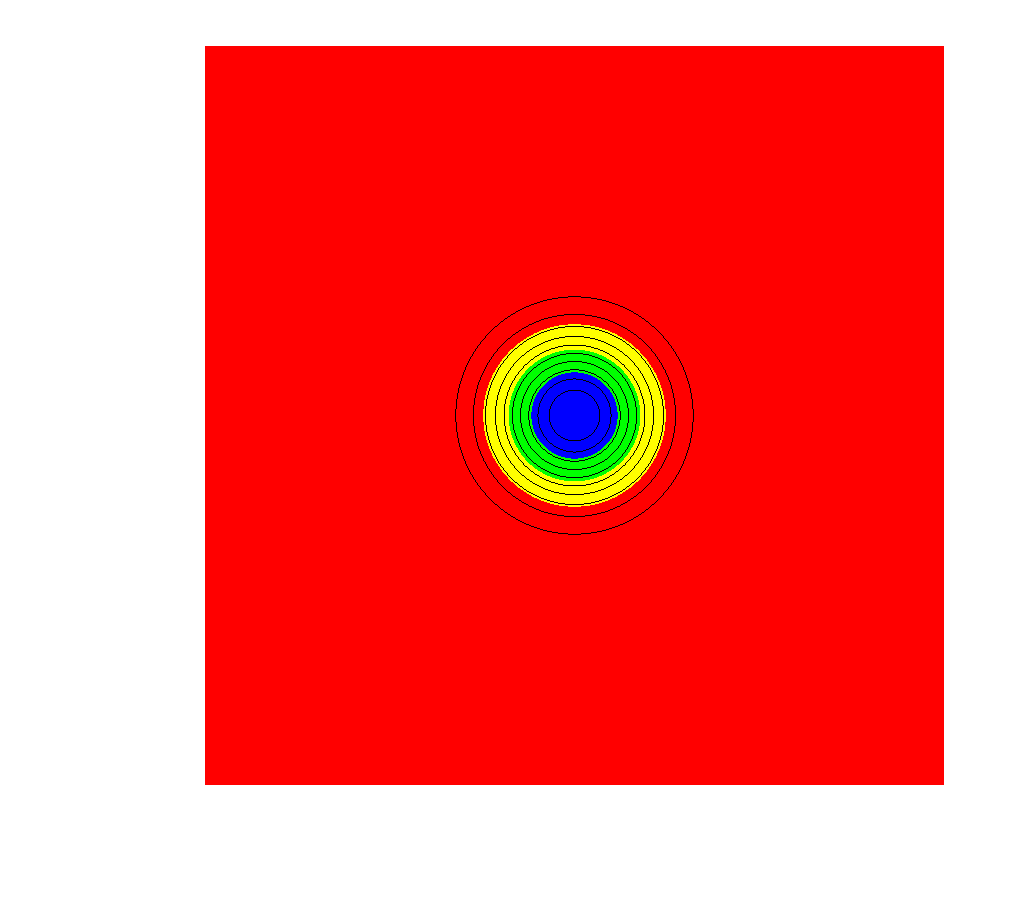} 
    		\caption{Initial Condition}
    		\label{fig:init_solution2}
  	\end{subfigure}%
	~
  	\begin{subfigure}[b]{0.32\textwidth}
		\includegraphics[width=\textwidth]{%
      		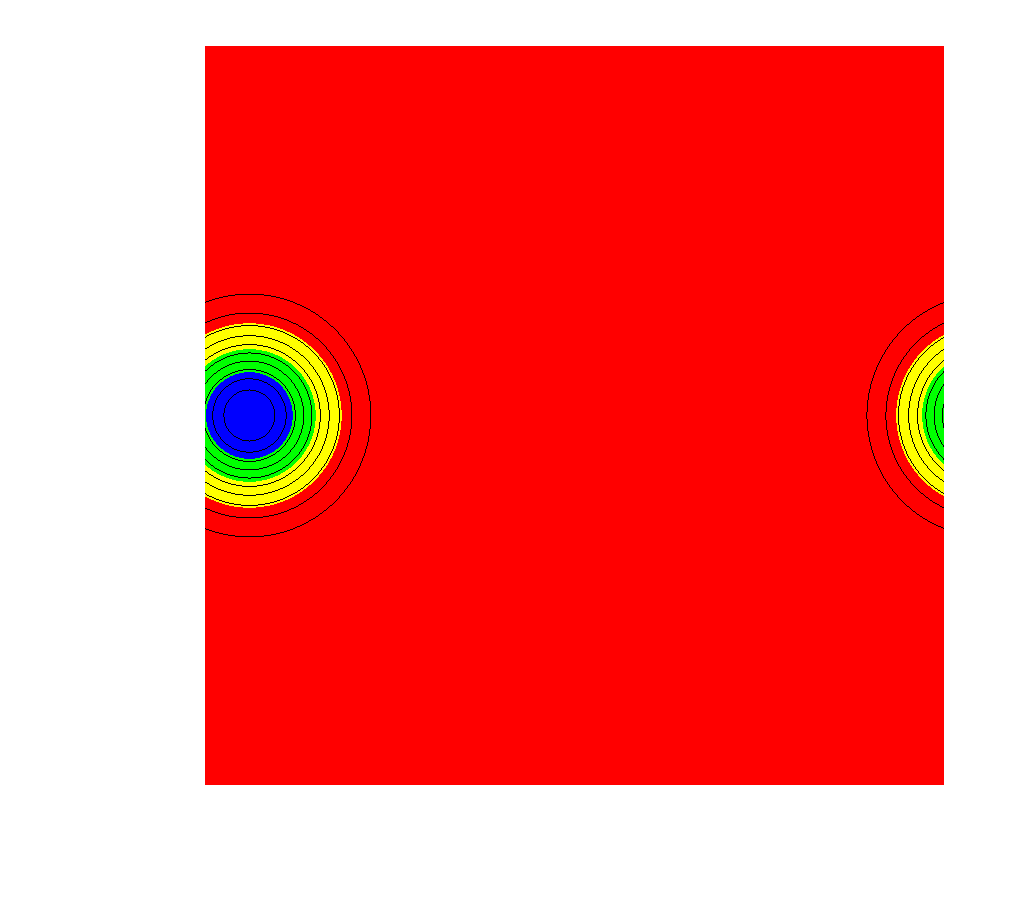} 
    		\caption{Numerical Solution after 10000 steps}
    		\label{fig:init_energy2}
  	\end{subfigure}%
	~
  	\begin{subfigure}[b]{0.32\textwidth}
		\includegraphics[width=\textwidth]{%
      		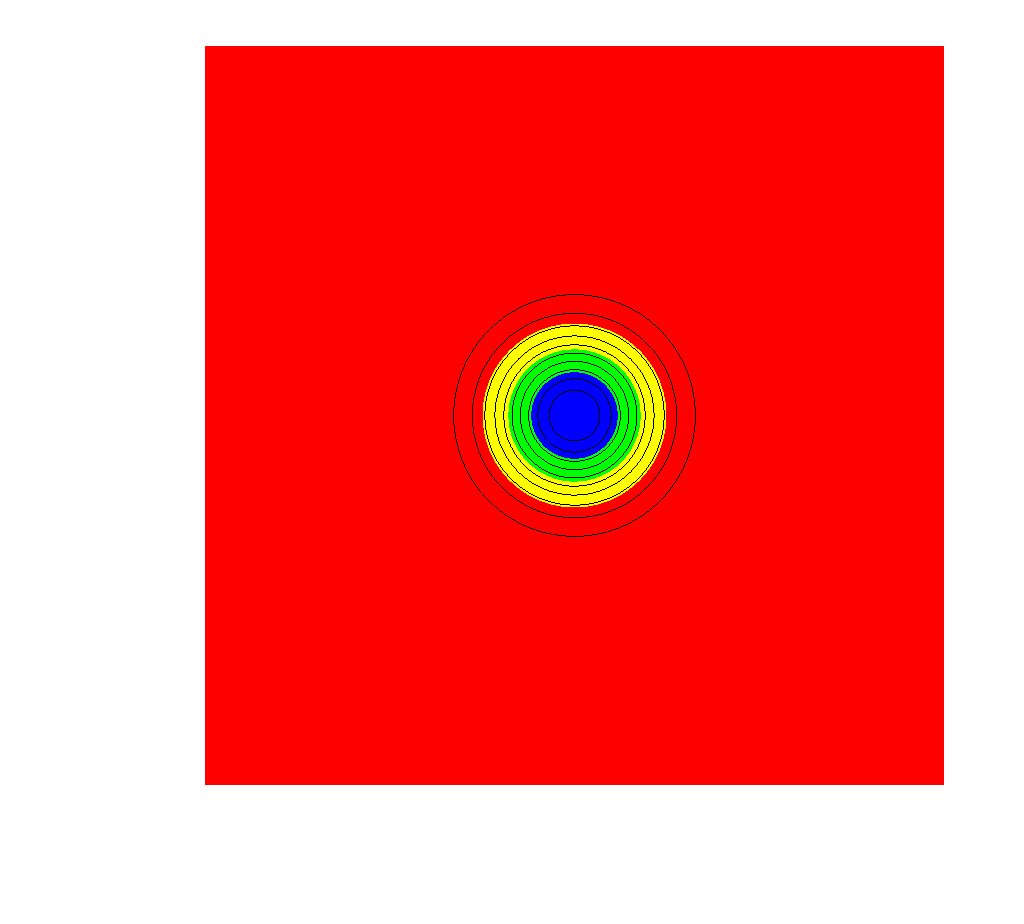} 
    		\caption{Numerical Solution at T=10}
    		\label{fig:init_ener}
	\end{subfigure}%
  	\caption{Moving vortex at different time levels}
  	\label{fig:init_approx2}
\end{figure}  
The center of the vortex is set in $(0,0)$. The selected parameters of the unperturbed flow are set to $u_{\infty}=1$, $v_{\infty}=0$ for the velocities, $p_{\infty}=1$ for the density. 
In this first simulation, we apply the pure Galerkin residual with entropy correction and $B1$ polynomials on an unstructured trianglular mesh. 
The vortex is moving to the right and at $T=10$ we reach again our starting point as can be seen in Figure \ref{fig:init_approx2}.  
The vortex is moving only to the right (it reduces to a simple advection), therefore, we have a classical smooth solution in this test case. 
From Theorem \ref{th_convergence}, a grid convergence investigation ensures that the errors decrease  to the analytical solution. 
We verify this result in Figure \eqref{fig:erorr}, where we plot the error behaviors in respect to the number of elements 
and clearly a grid convergence can be recognized. For a more specific investigation of the order of various different RD schemes and more experiments, we refer again to \cite{offner2020stability, michel2021spectral}.
\begin{figure}[tb]
	\centering 
	\begin{subfigure}[b]{0.32\textwidth}
		\includegraphics[width=\textwidth]{%
      		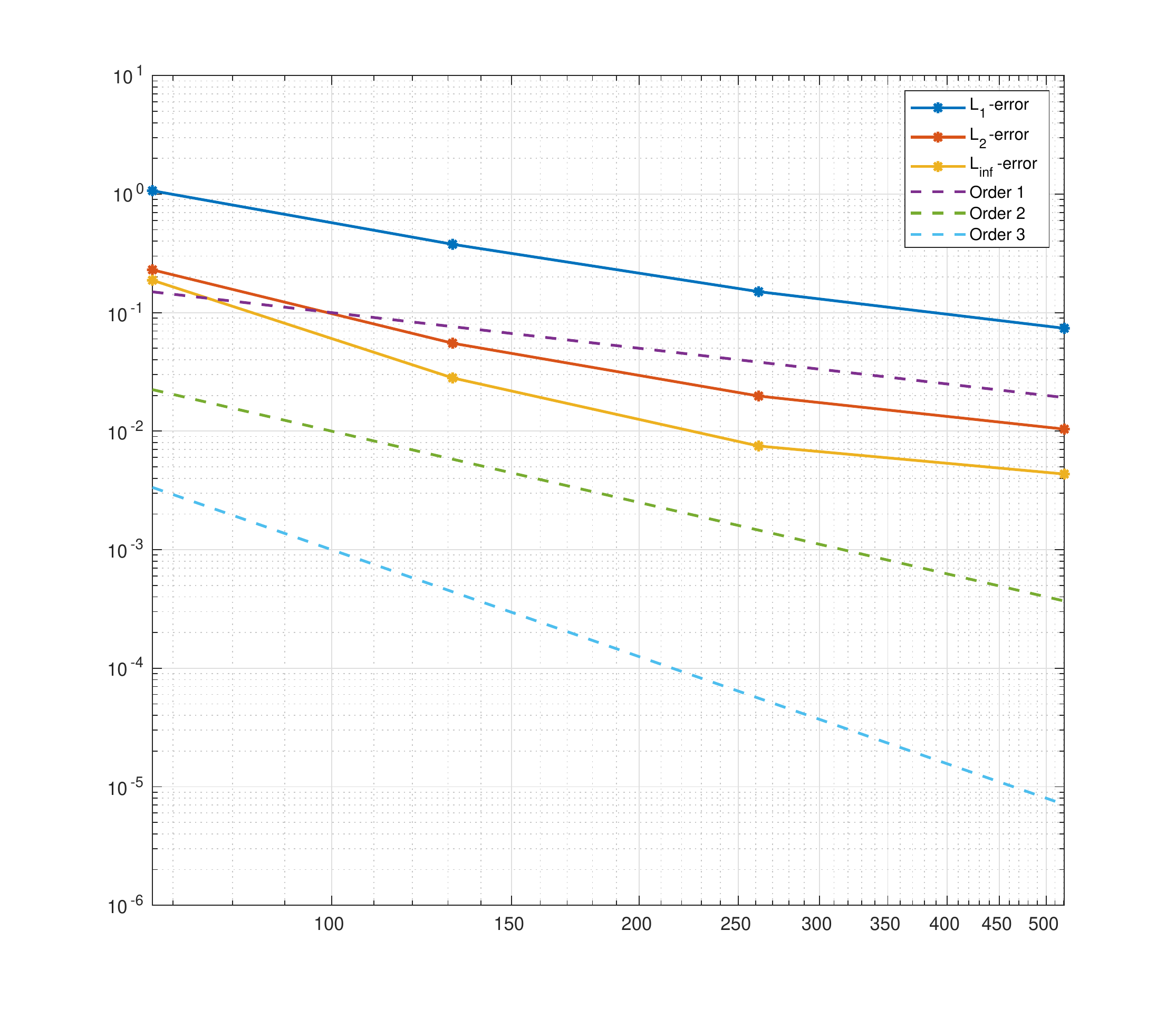} 
    		\caption{Density $\rho$-error behavior}
    		\label{fig:rho}
  	\end{subfigure}%
	~
  	\begin{subfigure}[b]{0.35\textwidth}
		\includegraphics[width=\textwidth]{%
      		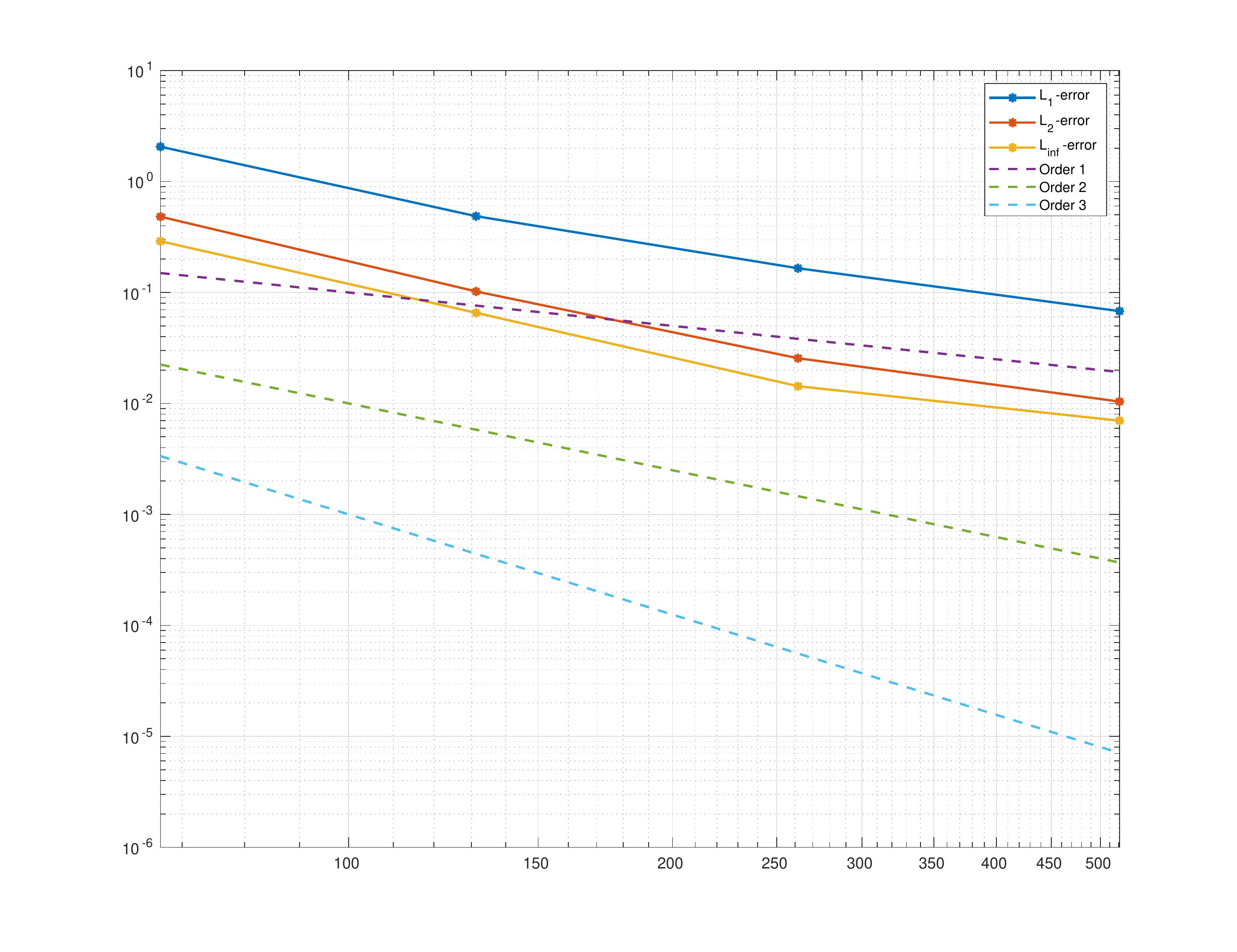} 
    		\caption{Velocity $u$-error behavior}
    		\label{fig:u}
  	\end{subfigure}%
	~
  	\begin{subfigure}[b]{0.33\textwidth}
		\includegraphics[width=\textwidth]{%
      		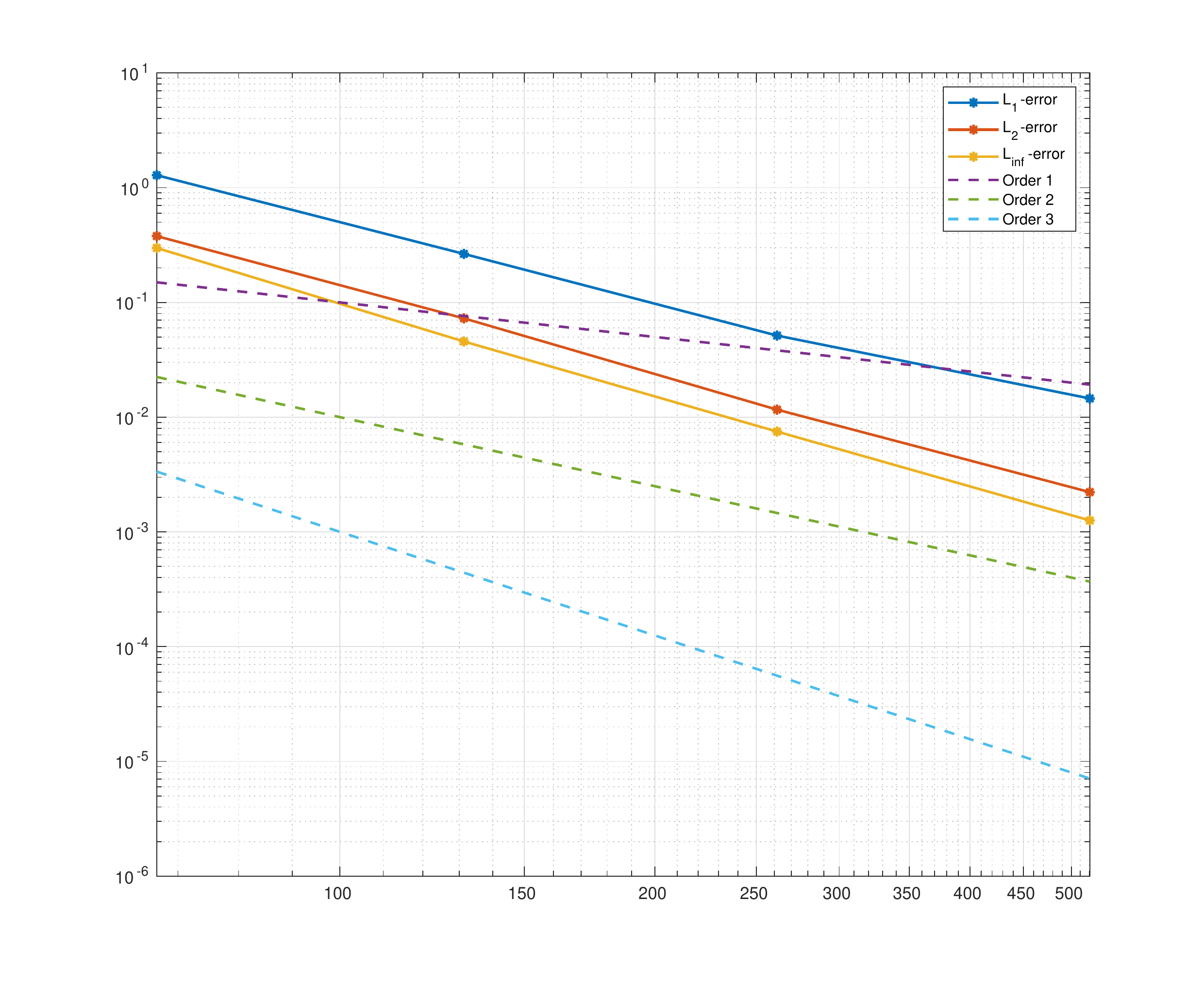} 
    		\caption{Pressure $p$-error behavior}
    		\label{fig:p}
	\end{subfigure}%
  	\caption{Error behaviour at $T=10$}
  	\label{fig:erorr}
\end{figure}

\section{Conclusion and Outlook} \label{sec_conclusion}

We have demonstrated a convergence analysis of the Euler equations via dissipative weak solutions for the general framework of residual distribution schemes including 
several high-order FE methods. Essential in our investigate was that the RD schemes ensure the underlying physical laws like positivity of internal energy and density and 
yields a consistent approximation of the Euler equations. We prove the consistency for every entropy dissipative RD schemes and so for all high-order finite element schemes  which 
can be interpreted in this framework. The proof used a weak BV estimate which is derived by the entropy dissipative property of the considered RD scheme. 
Due to the results of \cite{abgrall2018general, abgrall2019reinterpretation}, we can guaranteed  entropy dissipation of our considered RD scheme by the application of entropy correction terms 
as described in the mentioned literature. 
To guarantee the positivity of pressure and density, we have applied the MOOD approach and used as the parachute scheme the Lax-Friedrich residual which 
we have demonstrate that it is positivity preserving for both an explicit and  and implicit time integration method.  Our convergence analysis was restricted to the semi-discrete setting in the future, we plan to 
extend our investigations to the fully discrete framework \cite{abgrall2022relaxation, friedrich2019entropy}. A unifying analysis for several high-order FE methods will also be considered. \rev{Similar convergence results can be expected for other structure-preserving schemes, see \cite{kuzmin2020monolithic, kuzmin2022limiter, guermond2014second}.}
The concept of dissipative weak solutions is not restricted to the  Euler equations, but it has been also used for the compressible  Navier-Stokes equations \rev{\cite{feireisl2021numerics,zbMATH07229478, zbMATH07437238}, viscous magentohydrodynamics (MHD) \cite{bangwei2022convergence} and viscous multi-component flows \cite{zbMATH07559963,zbMATH07446443}}. However,  no analytical and numerical results 
exists -up to our knowledge- focussing on non-viscous MHD or  non-viscous multi-component/phase flows. Extensions of DW to those models will be considered in the future. In particular, the interpretation in multi-component 
seems very promising.  Here,  already in the description of the equations some of the quantities can be  expressed in terms of measures and so the concept of dissipative measure-valued/weak solutions seems even more beneficial.

\section{Appendix}\label{appendix}

\subsection*{Positivity property of the LxF residuals.}\label{positivity}

\subsubsection*{Explicit case}
In the following part, we repeat the proof about the positivity of the local Lax-Friedrich (LxF) (or Rusanov) residual following \cite{bacigaluppi2019posteriori}.
For the time-integration, we apply a simple Euler step first. 
Independent of the interpretation of the degrees of freedom, we obtain:
\begin{equation}
\label{remi:1}
\bU_\sigma^{n+1}=\bU_{\sigma}^n-\dfrac{\Delta t}{|C_{\sigma}|}\sum_{K|\sigma\in K} \Phi_{\sigma,\mathbf{x}}^{K,LxF}(\bU^n).
\end{equation}
For later, we can re-interpret  $\bU_\sigma^{n+1}$ using $ |K_\sigma|=\dfrac{|K|}{N_K}$: 
\begin{equation}
\label{remi:2}
\bU_\sigma^{n+1}=\sum_{K, \sigma\in K} \dfrac{|K_\sigma|}{|C_{\sigma}|} \bU_\sigma^{K,\star}
\text{ with }\bU_\sigma^{K,\star}=\bU_\sigma^n-\dfrac{\Delta t}{|K_\sigma|}\Phi_{\sigma,\mathbf{x}}^{K,LxF}(\bU^n). 
\end{equation}
The LxF described in \eqref{eq_LxF_residual} can be interpreted in the following two ways: 
\begin{enumerate}
\item The first version similar to the variational form
$\Phi_{\sigma,\mathbf{x}}^{K,LxF}(\bU^h)=\int_K \phi_\sigma \text{ div } \bbf^h(\bU^h) \;d\bx+\alpha_K\big (\bU_\sigma-\overline{\bU}\big )$, 
with $ \bbf^h(\bU^h)=\sum_{\sigma\in K}  \bbf_\sigma \phi_\sigma$,
\item where the second version is in   respect to the   nature of the degrees of freedom:
\begin{equation}\label{eq_second}
\Phi_{\sigma,\mathbf{x}}^{K,LxF}=\dfrac{1}{N_{K}}\int_{K}\text{ div } \bbf(\bU^h) \;d\bx+\alpha_K\big (\bU_\sigma-\overline{\bU}\big )
\end{equation}
\end{enumerate}
Here, we denote again by $\mean{\bU}$ the arithmetic mean.\\
Indeed, in the second interpretation it is important if we use Lagrange or Bernstein polynomials. 
The purpose of the following investigation is to estimate $\alpha_K$ in a way that we can ensure that the density and pressure remains positive after each time step. 
We will first consider the case of the Lagrange interpolation  and extend  it to Bernstein approximation afterwards.
\subsubsection*{Set of thermodynamical states}
Instead of working with the pressure, it is equivalent to focus on the positivity of the internal energies $e= E-\frac{1}{2}\frac{|\bm|^2}{\rho}$. 
In the following part, our goal is to have positive density and internal energies. We can define the set of admissible states including density, momentum and total energy. 
This set is convex under standard assumptions on the thermodynamics variables which holds for the here considered equations of state for standard perfect\footnote{It holds as well for stiffened gas.} gas. 
For Lagrange polynomials using nodal values it is straightforward to show that the density and the internal energy are positive. \\
Since the DOFs do not correspond to general point values in case of Bernstein polynomials, we will focus on this analysis.
Note that 
\begin{equation}
\begin{split}
\mathcal{K}_{th}'=\{ (\rho_\sigma, m_\sigma, E_\sigma)_{\sigma\in K} \text{ s. t. } & \rho=\sum_{\sigma\in K} \rho_\sigma B_\sigma\geq 0 \text{ on } K
\text{ and }  E-\frac{1}{2}\frac{|\bm|^2}{\rho}\geq 0,\\& \text{ with } E=\sum_{\sigma} E_\sigma B_\sigma \text{ and } \bm=\sum_{\sigma} \bm_\sigma B_\sigma\},
\end{split}
\label{set_convex_Bernstein}
\end{equation}
where we have denoted the momentum by $\bm=\rho \bu$, is convex and can be seen component wise. 
Instead of testing the inequalities for all $\bx\in K$, it is enough to apply  only  a finite set of points, e.g. Lagrange points. With a slight abuse of notations, we denote the resulting set again by $\mathcal{K}_{th}'$ and demonstrate that it is also convex. 
\begin{proof}
If the functions $\rho=\sum\limits_{\sigma} \rho_\sigma B_\sigma$ and $\rho'=\sum\limits_{\sigma'} \rho_{\sigma'} B_\sigma$ defined similarly are positive, and hence for any $\lambda\in [0,1]$,  the densities defined from $\bU=\lambda \bU+(1-\lambda)\bU'$ are positive on the simplex $K$.
The internal energy is a rational function of the conserved quantities.
We consider the mapping $\phi_e: (\rho, \bm, E)\mapsto E-\dfrac{1}{2}\dfrac{|\bm|^2}{\rho}.$ and demonstrate that the  internal energy $E-\dfrac{1}{2}\dfrac{|\bm|^2}{\rho}$ is a concave function of  the conservative variables $(\rho, \bm,E)$. 
 For  $\rho>0$,   $-\frac{|\bm|^2}{\rho}$ is obviously a concave function and   the internal energy is a sum of concave functions itself. \\
Hence, if $\bU$ and $\bU'$ belong to $\mathcal{K}_{th}'$, and $\lambda\in [0,1]$, the density function associated to $\lambda \bU+(1-\lambda)\bU'$ will be positive, and the internal energy is $\phi_e(\lambda \bU+(1-\lambda)\bU')$. Since $\phi_e$ is concave,
$$\phi_e(\lambda \bU+(1-\lambda)\bU')\geq \lambda\phi_e(\bU)+(1-\lambda)\phi_e(\bU')\geq 0,$$
and we can follow that $\lambda \bU+(1-\lambda)\bU'\in \mathcal{K}_{th}'$ holds and we obtain the desired result.
\end{proof}
Instead of working with  $\mathcal{K}_{th}'$ (which is hard to handle), we apply an even stronger condition. We consider 
the set
$$ \overline{\mathcal{K}}'_{th}=\bigg \{ \text{For all DOFS }\sigma, (\rho_\sigma, \bm_\sigma, E_\sigma ), \rho_\sigma\geq 0, E_\sigma-\frac{1}{2}\dfrac{|\bm_\sigma|^2}{\rho_\sigma}\geq 0\bigg \}.$$
Note that $$\overline{\mathcal{K}}'_{th}\subset \mathcal{K}'_{th}.$$
\begin{proof}
Thanks to the positivity of the Bernstein polynomials, and the Cauchy-Schwarz inequality, we have
$$\bigg ( \sum\limits_{\sigma\in K} m_{\sigma,i} B_\sigma\bigg )^2 \leq \bigg ( \sum\limits_{\sigma\in K}\rho_\sigma B_\sigma\bigg ) \bigg ( \sum\limits_{\sigma\in K}\dfrac{ m_{\sigma,i} ^2}{\rho_\sigma} B_\sigma\bigg ),$$
with $i=1,2$. Finally, we obtain
and then
$$\sum\limits_{\sigma\in K} E_\sigma B_\sigma -\frac{1}{2}\dfrac{\bigg ( \sum\limits_{\sigma\in K} m_{\sigma,1} B_\sigma\bigg )^2+
\bigg ( \sum\limits_{\sigma\in K} m_{\sigma,2} B_\sigma\bigg )^2}{\sum\limits_{\sigma\in K}\rho_\sigma B_\sigma}\geq \sum\limits_{\sigma\in K} \bigg ( E_\sigma-\frac{1}{2}\dfrac{|\bm_\sigma|^2}{\rho_\sigma}\bigg ) B_\sigma\geq 0$$
under the condition that $(\rho_\sigma, \bm_\sigma, E_\sigma)\in \overline{\mathcal{K}}'_{th}.$
\end{proof}

In the following, we specify now how we ensure from our LxF residuals that our approximated solution lies also in the next time step in $ \overline{\mathcal{K}}'_{th}$
\subsubsection*{Case of Lagrange interpolation}
First, we consider the Lagrange interpolation for simplicity. $\bU_\sigma=\bU(\sigma)$ is the evaluation of the solution at the Lagrange DOFs.
Please have in mind that they are defined in a simplex by their barycentric coordinates which are, for the degree $p$ and for triangles, $(\frac{i_1}{p+1},\frac{i_2}{p+1}, \frac{i_3}{p+1})$ with $i_1+i_2+i_3=p+1$. We obtain them for quads and hex elements by simply considering tensorisation of the 1D Lagrange points and the 3D case can be handle analogously. 

\subsubsection*{The one-dimensional case} 
We start our investigation for the one-dimensional setting for simplicity since the extension to two-dimension will be done along the edges in a similar matter. 
We rephrase the proof of Perthame and Shu \cite{PerthameShu1996positivity} for the classical LLF method and have 
\begin{equation}\label{Proofposi_lxf}
\bU_i^{n+1}=\bU_i^n-\frac{\Delta t}{\Delta x} \left[ \widehat{\bbf}(\bU_{i+1},\bU_i)-\widehat{\bbf}(\bU_i,\bU_{i-1})\right]
\end{equation}
with the LLF fluxes $\widehat{\bbf}$. 
Further, we denote now by $K$ the interval $[x_i, x_{i+1}]$ and introduce the splitting of
$$\bU_t+\bbf(\bU)_x=0$$
via
\begin{equation}
\label{remi:split:1}
\bU_t+\big (\bbf(\bU)+\nu\bU\big )_x=0, ~\text{and}~ \bU_t+\big (\bbf(\bU)-\nu\bU\big )_x=0.
\end{equation}
If $\nu=\max\limits_{x\in K} ||u(x)||+c(x)$ (with velocity $u$ and sound speed $c$), the local Lax-Friedrich method can be recast as a combination of the Godunov scheme and the downwind scheme, i.e. the left and right equations in \eqref{remi:split:1}.
Hence, we can reinterpret the value $\bU_i^{n+1}$ as the average of 
$$
\widetilde{\bU}=\bU_i^n-\frac{\Delta t}{\Delta x} \left[ \left( {\bbf}(\bU_i^n)+\nu\bU_i^n \right)-\left( \bbf(\bU_{i-1}^n)+\nu\bU_{i-1}^n \right) \right]$$
and
$$
\widetilde{ \widetilde{\bU} }= \bU_i^n-\frac{\Delta t}{\Delta x} \left[ \left({\bbf}(\bU_{i+1}^n )-\nu\bU_{i+1}^n \right)-\left( \bbf(\bU_{i}^n) -\nu\bU_{i}^n \right) \right].$$
In case that the states $\bU_i^n$, $\bU_{i+1}^n$ and $\bU_{i-1}^n$ belong to the convex set $\mathcal{K}_{th}$ defined through
$$\mathcal{K}_{th}= \left\{ (\rho, \rho u, E), \rho\geq 0, E-\frac{1}{2}\rho u^2\geq 0 \right\},$$
$\bU_i^{n+1}$ will belong as well to  $\mathcal{K}_{th}$. We can finally conclude that for both  the Lagrangian interpolation and the Bernstein reconstruction, the Rusanov scheme \eqref{Proofposi_lxf} for the one-dimensional case preserves the convex sets $\mathcal{K}_{th}$ for the Lagrange interpolation and $\mathcal{K}_{th}'$ for the Bernstein reconstruction.
\subsubsection*{The multi-dimensional case}
In the following, we extend now the investigation from before along the edges and estimate $\alpha_K$ to guarantee that we remain in our convex set. 
The residuals can be written using the first version of the LxF residuals:
\begin{equation}
\begin{split}
\Phi_{\sigma,\mathbf{x}}^{K,LxF}&=\int_K\phi_\sigma \; \text{ div }\bbf\; d\bx+\alpha \big ( \bU_\sigma-\bar\bU\big )
\\&=\sum_{\sigma'\in K} \left[ \left( \int_K \phi_\sigma\nabla \phi_{\sigma'}d\bx \right) \cdot \bbf_{\sigma'}+\frac{\alpha}{N_K} \left( \bU_\sigma-\bU_{\sigma'} \right) \right] \\ & =\sum_{\sigma'\in K} \left[  2\left( \int_K \phi_\sigma\nabla\phi_{\sigma'}d\bx \right) \cdot\dfrac{\bbf_\sigma+\bbf_{\sigma'}}{2}+\frac{\alpha_K}{N_K} \left( \bU_\sigma-\bU_{\sigma'} \right) \right]
\end{split}
\end{equation}
due to $\sum\limits_{\sigma'\in K}\int_K \phi_\sigma\nabla\phi_{\sigma'}d\bx=\int_K \phi_\sigma \nabla (1)\; d\bx=0$.
In addition, we can rewrite \eqref{remi:2} as:
\begin{equation*}
\begin{split}
\bU_\sigma^\star&=\bU_{\sigma}^n-\dfrac{\Delta t}{|K_\sigma|}\sum_{\sigma'\in K} \left[ \left( 2\int_K \phi_\sigma\nabla\phi_{\sigma'} d\bx \right) \cdot\dfrac{\bbf_\sigma+\bbf_{\sigma'}}{2}+\frac{\alpha}{N_K}  \left( \bU_\sigma-\bU_{\sigma'} \right) \right]\\
&=\dfrac{1}{N_{K}} \sum\limits_{\sigma'\in K}\left[ \bU_{\sigma}^n-\dfrac{\Delta t}{|K_\sigma|} \;\omega_{\sigma\sigma'}\cdot \dfrac{\bbf_\sigma+\bbf_{\sigma'}}{2}+\alpha_K \left( \bU_\sigma-\bU_{\sigma'} \right) \right]
\end{split}
\end{equation*}
where the vector $\omega_{\sigma\sigma'}=2N_K\int_K \phi_\sigma\nabla\phi_{\sigma'}\; d\bx.$ can be interpreted as a scaled normal since $\int_K\phi_\sigma=\frac{|K|}{N_K}$ holds.
We can derive now the estimate for  $\alpha_K$. It has to hold that 
$\alpha_K\geq \max_{\bx\in K} \rho(\mathbf{A(U(\bx))}\cdot\omega_{\sigma\sigma'})$,
where for any vector $\bn=(n_x,n_y)$, $\mathbf{A(U)}\cdot \bn=\dpar{f_1}{\bU}(\bU) n_x+\dpar{f_2}{\bU}(\bU) n_y$, and $f_1$ (resp $f_2$) is the $x$- (resp. $y$-) component of the flux $\bbf$. 
\begin{remark}[Version 2]
The consideration from above can be directly extended to the second version of the interpretation of the LxF residuals \eqref{eq_second} by simple realizing 
$$\int_K \mathrm{ div }\; \bbf \; d\bx =\sum_{\sigma\in K} \left( \int_K\nabla \phi_\sigma\; d\bx \right) \cdot \bbf_\sigma$$
holds. An analogous estimate for $\alpha_K$ can be derived. 
\end{remark}
\subsubsection*{Bernstein Polynomials}
Next, we extend the estimation from before in case that Bernstein polynomials are used. The basic idea now is to transfrom everything back to the Lagrange case and use the result from above. 
When applying Bernstein polynomials and their DOFs, we have at the at the Lagrange degrees of freedom (denoted by $\sigma_L$ in the following)
$$U(\sigma_L)=\sum_{\sigma}U_\sigma B_\sigma(\sigma_L), \text{ with }B_\sigma(\sigma_L)\geq 0, \text{ and  }\sum_\sigma B_\sigma(\sigma_L)=1.$$
Hence we can define a  linear mapping
$M=\bigg(B_\sigma(\sigma_L)\bigg )$ from the Bernstein DOFs to the Lagrange DOFs.
If  we now consider our scheme 
$$\bU_\sigma^{K,\star}=\bU_\sigma^n-\dfrac{\Delta t}{|K_\sigma|}\Phi_{\sigma,\mathbf{x}}^{K,LxF}(\bU^n)$$ for the Bernstein DOFs, 
we can re-write it at the Lagrange points via 
$$\begin{pmatrix}
\bU(\sigma_L)^{K,\star}\end{pmatrix}=\begin{pmatrix}\bU^n(\sigma_L)\end{pmatrix}-\dfrac{\Delta t}{|K_\sigma|}M\begin{pmatrix}\Phi_{\sigma,\mathbf{x}}^{K,LxF}(\bU^n)\end{pmatrix}. 
$$
Since we approximate the flux as
$\bbf=\sum\limits_{\sigma\in K} \bbf_\sigma B_\sigma$, we realize that  
$M\begin{pmatrix}\Phi_{\sigma,\mathbf{x}}^{K,LxF}(\bU^n)\end{pmatrix}$ is nothing more than the LxF residuals computed with the Lagrange interpolation
$\bbf=\sum\limits_{\sigma_L\in K}\bbf(\sigma_L)\phi_{\sigma_L}$ where $\phi_{\sigma_L}$ denotes the Lagrange polynomial for $\sigma_L$. 
The estimation for $\alpha_K$ can be derived now following our previous analysis. At all, we can ensure the positivity of the density and the internal energy at the Lagrange points. 

\subsubsection*{Implicit case}

Here, we give now an extension of above consideration if we apply an implicit time-integration method (implicit Euler) instead. This is also done for the first time and will be essential in future work. 
The RD version reads like
\begin{equation}
\label{remi:1}
\bU_\sigma^{n+1}=\bU_{\sigma}^n-\dfrac{\Delta t}{|C_{\sigma}|}\sum_{K|\sigma\in K} \Phi_{\sigma,\mathbf{x}}^{K,LxF}(\bU^{n+1}).
\end{equation}
If we approximate the flux $\bbf$ by:
$\bbf\approx \sum\limits_{\sigma\in K}\bbf_\sigma\phi_\sigma$
as a polynomial of degree $p$. We can write the residual as
$
 \Phi_{\sigma,\mathbf{x}}^{K,LxF}=\sum\limits_{\sigma'\in K}c_{\sigma\sigma'}\:\bbf_\sigma
$
with
$$c_{\sigma\sigma'}=\left \{\begin{array}{ll}
\int_K\phi_\sigma\text{div }\phi_\sigma \;d\bx+\alpha_K\frac{N_K-1}{N_K},& \text{ if } \sigma=\sigma',\\
~\\
\int_K\phi_{\sigma'}\text{div }\phi_\sigma \;d\bx-\alpha_K\frac{1}{N_K}, & \text{ else.}
\end{array}\right .$$
Note that $\sum\limits_{\sigma'\in K} c_{\sigma\sigma'}=0.$
We see that if 
\begin{equation}
\label{condition}\alpha_K\geq \max\limits_{\sigma, \sigma'\in K}\bigg \vert \int_K\phi_\sigma\text{div }\phi_{\sigma'} \;d\bx\bigg \vert,\end{equation}
then $c_{\sigma\sigma'}\geq 0$.
We use this knowledge to demonstrate the positivity of the density for the Euler equations as an example. 
The RD scheme  is:
$$\vert C_\sigma \vert \big ( \bU^{n+1}_\sigma-\bU^n_\sigma\big ) +\Delta t \sum_{K, \sigma \in K} \Phi_\sigma^K(\bU^{n+1})=0.$$
Due to a result from Forestier and Gonzales-Rodelas \cite{forestier2009}, the system is solvable. 
 Writing it for the density, we obtain:
$$\vert C_\sigma \vert \big ( \rho^{n+1}_\sigma-\rho^n_\sigma\big ) +\Delta t \sum_{K, \sigma \in K} \big (\Phi_\sigma^K\big )^\rho(\bU^{n+1})=0,$$
with
$$\big (\Phi_\sigma^K\big )^\rho(\bU^{n+1})=\sum_{\sigma'\in K} \big (c_{\sigma\sigma'}\cdot \bU_\sigma^{n+1}\big ) \rho_\sigma^{n+1},$$
so that we have
\begin{equation*}
\big (|C_\sigma|+\Delta t \sum_{K, \sigma\in K} c_{\sigma\sigma}\big ) \rho_\sigma^n+\sum\limits_{\sigma'\neq \sigma} d_{\sigma\sigma'}\rho_{\sigma'}^{n+1}=|C_{\sigma}|\rho_\sigma^n
\end{equation*}
with
\begin{equation*}
d_{\sigma\sigma}=|C_\sigma|+\Delta t \sum_{K, \sigma\in K} c_{\sigma\sigma}
\end{equation*}
and
\begin{equation*}
d_{\sigma\sigma'}=\Delta t\sum_{K, \sigma \in K}c_{\sigma\sigma'}.
\end{equation*}
We note that under the condition \eqref{condition}, $d_{\sigma\sigma'}<0$ while $d_{\sigma\sigma}>0$ independantly of $\Delta t$ because
$$d_{\sigma\sigma}=\vert C_\sigma\vert +\Delta t\sum_{K, \sigma\in K} c_{\sigma\sigma}=
\vert C_\sigma\vert +\Delta t\bigg ( \int\phi_\sigma\nabla\phi_\sigma\; d\bx+\sum_{K, \sigma\in K}\frac{N_K-1}{N_K}\alpha_K\bigg )=
\vert C_\sigma\vert +\Delta t\sum_{K, \sigma\in K}\frac{N_K-1}{N_K}\alpha_K>0.
$$
If $\rho^0$ has a compact support, then $\rho^{n}$ has a compact support.
We see that $\rho^{n+1}$ is obtained (if one know the velocity) by a linear system of the type $M\rho^{n+1}=\rho^n$ where $M$ is a $M$ matrix. Therefore, it follows that the inverse is positive definite. 


\subsubsection*{Case where  the flux is not interpolated}
In that case, we assume that the residual writes:
\begin{equation}
\label{remi:new:1}
\Phi_{\sigma,\mathbf{x}}^{K,LxF}(\bU^h)=-\int_K \nabla \varphi_\sigma\cdot \bbf(\bU^h)\; \;d\bx+\int_{\partial K}\varphi_\sigma\bbf(\bU^h)\cdot\bn\; d\gamma +\alpha_K\big (\bU_\sigma-\overline{\bU}\big )
\end{equation}
with the forward Euler time stepping.
The difficulty is that one cannot any longer rely on the one dimensional estimates as before. To study this case, we rely on the geometrical technique of Wu and Shu \cite{WuShu}\footnote{In this paper, noticing that all the "interesting" admissibility domains  are  convex and are the intersection of domains of the type $\{\psi_l>0\}$, where  $\psi_l$ is concave w.r.t. the conserved variables. They show how to exploit that a convex domain is the intersection of half spaces and demonstrate that a scheme is invariant domain preserving. In particular, to make sure that the internal energy $e=E-\tfrac{1}{2}\frac{\bm^2}{\rho}$ is positive, it is enough to check that for any velocity vector $\bv$, $(1, -\bv, \tfrac{\bv^2}{2})^T U\geq 0$. This is what we use here.}.
It is enough to show that \eqref{remi:new:1} preserves the positivity of the density and that a combination of the density, momentum and energy residual satisfies a positivity bound.
First we write
$$\bU_\sigma^{n+1}=\sum\limits_{K, \sigma\in K} \frac{|K_\sigma|}{|C_\sigma|} \bU_\sigma^{n+1,K}$$
with
$|K_\sigma|=\tfrac{|K|}{\# dof}$ and
$$\bU_\sigma^{n+1,K}=\bU_\sigma^n-\frac{\Delta t}{|K_\sigma|}\Phi_{\sigma,\mathbf{x}}^{K,LxF}(\bU^n)$$
so we will focus our attention on $\bU_\sigma^{n+1,K}$. And for simplicity, we restrict ourself to the Lagrange case.

Let us look at the positivity of the density. We have
$$\rho_\sigma^{n+1,K}=\rho_\sigma^n-\frac{\Delta t}{|K_\sigma|} \bigg ( -\int_K \nabla\varphi_\sigma \big ( (\rho  \bu)^h-\rho_\sigma  \bu_\sigma \big ) \; d\bx+
\int_{\partial K} \varphi_\sigma\big ( (\rho  \bu)^h-\rho_\sigma  \bu_\sigma \big )\cdot \bn\; d\gamma+\alpha_K (\rho_\sigma -\overline{\rho})\bigg ).$$
Note that it is in writing $(\rho\bu)_\sigma=\rho_\sigma \bu_\sigma$ that we say we consider Lagrange interpolant.

Next we introduce 
$$\mathbf{N}_{\sigma\sigma'}=-\int_K \nabla\varphi_\sigma \varphi_{\sigma'} \; d\bx+\int_K \varphi_\sigma\varphi_{\sigma'}\bn\; d\gamma,$$
so
$$\rho_\sigma^{n+1,K}=\rho_\sigma^n-\frac{\Delta t}{|K_\sigma|} \sum_{\sigma'\in K} \bigg ( \big ( \rho_{\sigma'} \bu_{\sigma'}-\rho_\sigma  \bu_\sigma \big )\cdot \mathbf{N}_{\sigma\sigma'}
+\alpha_K (\rho_\sigma -\overline{\rho})\bigg ).$$
Then we see that
$$\rho_\sigma^{n+1,K}= \big (1-\frac{\Delta t}{|K_\sigma|} \big (-\bu_\sigma\cdot \mathbf{N}_{\sigma\sigma'}+\alpha_K\big )\big )\rho_\sigma^n+\frac{\Delta t}{|K_\sigma|}\sum_{\sigma'\neq\sigma }\bu_{\sigma'}\cdot \big (\mathbf{N}_{\sigma\sigma'}+\frac{\alpha_K}{\# dof}\big )\rho_{\sigma'}$$
so that positivity is met if 
$$\alpha_K>\max\limits_{\bx\in K}\big ( \Vert \bu^n(\bx)\Vert +a(\bx))\times\max\limits_{\sigma,\sigma'\in K}  \Vert \mathbf{N}_{\sigma\sigma'}\Vert.$$

Let $\bv^\star\in \R^d$ an arbitrary vector. We compute
$$S=\frac{\Vert\bv^\star\Vert^2}{2}\Phi_{\sigma, \rho}-\bv^\star \cdot \Phi_{\sigma,\mathbf{m}}+\Phi_{\sigma, E}.$$ 
For $\mathbf{N}=\bn$ or $\nabla\varphi_\sigma$, we have
$$
D_{\mathbf{N}}:=\frac{\Vert\bv^\star\Vert^2}{2}\rho \bu\cdot\mathbf{N}-\bv^\star\cdot\big ( \rho \bu\cdot\mathbf{N}\bu +p\mathbf{N}\big )+\bu\cdot\mathbf{N} (E+p)=p(\bu-\bv^\star)\cdot \mathbf{N}+\bu\cdot\mathbf{N}\big ( \frac{\rho}{2}\Vert \bu-\bv^\star\Vert^2+\rho e\big ),$$
so that
$$D_{\mathbf{N}}\leq p \vert (\bu-\bv^\star)\cdot \mathbf{N}\vert + \vert \bu\cdot\mathbf{N}\vert \big ( \frac{\rho}{2}\Vert \bu-\bv^\star\Vert^2+\rho e\big ).$$
Then we show, for a perfect gas, that
$$p \vert (\bu-\bv^\star)\cdot \mathbf{N}\vert\leq \Vert \mathbf{N}\Vert \; a\big ( \frac{\rho}{2}\Vert \bu-\bv^\star\Vert^2+\rho e\big ).$$
For this, we first have
$$p \vert (\bu-\bv^\star)\cdot \mathbf{N}\vert\leq p\Vert \bu-\bv^\star \Vert \; \Vert \mathbf{N}\Vert, $$ and this amounts to showing that
$$p\Vert \bu-\bv^\star \Vert \; \Vert \mathbf{N}\Vert\leq \Vert \mathbf{N}\Vert \; a\big ( \frac{\rho}{2}\Vert \bu-\bv^\star\Vert^2+\rho e\big )$$ which means that the quadratic polynomial
$$ a\big ( \frac{\rho}{2}\ X^2+\rho e\big )-p X$$ is always positive, i.e. (since the product of roots and the trace are both positive) that the discriminant
$$\Delta =p^2-2\rho^2 \; a^2\;e$$ is negative. For a perfect gas, we have
$$\Delta =p^2-2\rho\frac{\gamma p}{\rho }\times \frac{p}{\gamma -1}=p^2\big ( 1-2\frac{\gamma}{\gamma-1}\big )=-p^2\frac{\gamma+1}{\gamma-1}<0.$$
All this shows that
$$D_{{\mathbf{N}}}\leq \Vert \mathbf{N}\Vert \times \big ( \Vert \bu(\bx)\Vert +a(\bx)\big )\times \big ( \frac{\rho}{2}\Vert \bu-\bv^\star\Vert^2+\rho e\big ).$$
Denoting by $W:=\frac{\Vert\bv^\star\Vert^2}{2}\rho-\bv^\star \cdot\mathbf{m}+E,$  and
$D_{\mathbf{N}}(\sigma)$ the value of $D_{\mathbf{N}}$ evaluated for $\sigma$
we see that, as for the density, 
$$W_\sigma^{n+1,\star}=W_\sigma^n-\dfrac{\Delta t}{|K_\sigma|}\bigg (-\int_K \big (D_{\nabla \varphi_\sigma}-D_{\nabla \varphi_\sigma}(\sigma)\big )\; d\bx+\int_K \big (D_{\bn}-D_{\bn}(\sigma)\big )\varphi_\sigma \; d\bx+\alpha_K \big ( W_\sigma-\overline{W}\big )\bigg )$$
Then we can use the same technique as for the density.

\subsubsection*{Entropy production}
The question is starting from any version of the Lax-Friedrich residual, wether or not we can have an entropy inequality, and more important a bound on the entropy production.  We will focus on the form \eqref{remi:new:1} which is the one used in the paper.
For the implicit version, there is not much to do since 
$$V(\bU)^T(\bU-\bU')\leq S(\bU)-S(\bU').$$
The explicit case is more involved. First,
$$S(\bU^{n+1}_\sigma)\leq \sum_{K, \sigma\in K}\frac{|K_\sigma|}{|C_\sigma|} S(\bU^{n+1,K}_\sigma),$$
then  (where $V^T_\sigma$ is  evaluated at time $t_n$)
\begin{equation*}
\begin{split}
S(\bU^{n+1,K}_\sigma)-S(\bU^n_\sigma)&=
V^T_\sigma(\bU^n_\sigma)^T( \bU_\sigma^{n+1,K})-S(\bU^n_\sigma))\\
&\qquad +\frac{1}{2}\int_0^1A_0(t\bU_\sigma^{n+1,K}+(1-t)\bU^n_\sigma) \; dt\cdot \big (\bU_\sigma^{n+1,K}-\bU^n_\sigma ,\bU_\sigma^{n+1,K}-\bU^n_\sigma\big )\\
&=\frac{\Delta t}{|K_\sigma|} V_\sigma^T\Phi_\sigma^{LxF}(\bU^n)\\
&\qquad+\frac{1}{2}\bigg ( \int_0^1A_0(t\bU_\sigma^{n+1,K}+(1-t)\bU^n_\sigma)\; dt\bigg )\cdot \big (\bU_\sigma^{n+1,K}-\bU^n_\sigma, \bU_\sigma^{n+1,K}-\bU^n_\sigma\big ) 
\end{split}
\end{equation*}
To simplify the writing, we  write $A_{0,\sigma\sigma'}^{n+1/2}=\int_0^1A_0(t\bU_\sigma^{n+1,K}+(1-t)\bU_\sigma^n)\; dt$. We get
$$V_\sigma^T\Phi_\sigma^{LxF}(\bU^n)=\bV_\sigma^T\bigg (-\int_K\nabla\varphi_\sigma \cdot \bbf(\bU^n)\; d\bx+\int_{\partial K}\varphi_\sigma\bbf(\bU^n)\cdot \bn\; d\gamma\bigg )+
\frac{\alpha_K}{\#dof}\sum_{\sigma'}V^T\big ( \bU^n_\sigma-\bU^n_{\sigma'}\big ).$$
Denoting 
$$\Psi_\sigma=\bV_\sigma^T\bigg (-\int_K\nabla\varphi_\sigma \cdot \bbf(\bU^n)\; d\bx+\int_{\partial K}\varphi_\sigma\bbf(\bU^n)\cdot \bn\; d\gamma\bigg ),$$
we see that
$$\sum_{\sigma\in K}\Psi_\sigma=\int_{\partial K}\bbg(\bU^n)\cdot\bn\; d\gamma$$ so these are proper residuals, and using again the convexity of the entropy,
we get
$$V_\sigma^T\big ( \bU^n_\sigma-\bU^n_{\sigma'}\big )=S(\bU^n_\sigma)-S(\bU^n_{\sigma'})+\frac{1}{2}\int_0^1 A_0(t\bU^n_\sigma+(1-t)\bU^n_{\sigma'})\; dt \cdot \big ( \bU^n_\sigma-\bU^n_{\sigma'}, \bU^n_\sigma-\bU^n_{\sigma'})$$
In the following, we set
$$A_{0,\sigma,\sigma'}^n=\int_0^1 A_0(t\bU^n_\sigma+(1-t)\bU^n_{\sigma'})\; dt.$$

In the end we get
\begin{equation*}
\begin{split}
S(\bU_\sigma^{n+1,K})&=S(\bU^n_\sigma)-\frac{\Delta t}{|K_\sigma|} \bigg ( \Psi_\sigma+\frac{\alpha_K}{\#dof}\sum_{\sigma'\in K}\big (S(\bU^n_\sigma)-S(\bU^n_{\sigma'})\big )\bigg )\\
&-
\frac{\alpha_K\Delta t}{2|K_\sigma|}\sum_{\sigma'\in K} A_{0, \sigma,\sigma'}^n \cdot \big ( \bU^n_\sigma-\bU^n_{\sigma'}, \bU^n_\sigma-\bU^n_{\sigma'})\\
&\quad +
\frac{\alpha_K\Delta t^2}{2|K_\sigma|^2}\sum_{\sigma'\in K} A_{0, \sigma,\sigma'}^{n+1/2} \cdot \big (\Phi_\sigma^{LxF,K}(\bU^n),\Phi_\sigma^{LxF,K}(\bU^n) ).
\end{split}
\end{equation*}
Denoting by 
$$\mathcal{D}(\Delta t)=-\sum_{\sigma'\in K} A_{0, \sigma,\sigma'}^n \cdot \big ( \bU^n_\sigma-\bU^n_{\sigma'}, \bU^n_\sigma-\bU^n_{\sigma'})+\frac{\Delta t}{2|K_\sigma|}\sum_{\sigma'\in K} A_{0, \sigma,\sigma'}^{n+1/2} \cdot \big (\Phi_\sigma^{LxF,K}(\bU^n),\Phi_\sigma^{LxF,K}(\bU^n) ), $$
we see that $\mathcal{D}(0)=0$, and for non constant state in $T$, $\dpar{\mathcal{D}}{ t}(0)<0$, so for $\Delta t$ small enough, we have
$\mathcal{D}(\Delta t)\leq 0$ and 
$$\Vert \mathcal{D}(\Delta t)\Vert \leq C \sum_{\sigma'} \Vert \bU^n_\sigma-\bU^n_{\sigma'}\Vert^2$$ where $C$  depends on $\Delta t $ and $\bU^n$, and this can be interpreted as
$$\Vert \mathcal{D}(\Delta t)\Vert \leq C h^2 \int_{\partial K}\Vert \nabla \bU^n\Vert^2 \; d\gamma$$
because $\alpha_K=O(h^{d-1})$, so it is absorbed in the integral over the boundary.

\section*{Acknowledgements}
R. A. has been supported by SNF grant 200020\_204917 "
Structure preserving and fast methods for hyperbolic systems of conservation laws". 
M.L.-M. has been founded by the German Science Foundation (DFG) under the collaborative research projects TRR SFB 165 (Project A2) and TRR SFB 146 (Project C5).
M.L.-M. and P.\"O. also gratefully acknowledge support of the Gutenberg Research College,  JGU Mainz.

\bibliographystyle{abbrv}
\bibliography{literature}

\begin{thebibliography}{10}

\bibitem{abgrall2017high}
R.~{Abgrall}.
\newblock {High order schemes for hyperbolic problems using globally continuous
  approximation and avoiding mass matrices}.
\newblock {\em {J. Sci. Comput.}}, 73(2-3):461--494, 2017.

\bibitem{abgrall2018general}
R.~Abgrall.
\newblock A general framework to construct schemes satisfying additional
  conservation relations. {Application} to entropy conservative and entropy
  dissipative schemes.
\newblock {\em J. Comput. Phys.}, 372:640--666, 2018.

\bibitem{abgrall2011construction_2}
R.~{Abgrall}, A.~{Larat}, and M.~{Ricchiuto}.
\newblock Construction of very high order residual distribution schemes for
  steady inviscid flow problems on hybrid unstructured meshes.
\newblock {\em {J. Comput. Phys.}}, 230(11):4103--4136, 2011.

\bibitem{abgrall2022relaxation}
R.~Abgrall, E.~L. M\'el\'edo, P.~\"Offner, and D.~Torlo.
\newblock Relaxation {Deferred} {Correction} {Methods} and their {Applications}
  to {Residual} {Distribution} {Schemes}.
\newblock {\em The SMAI Journal of computational mathematics}, 8:125--160,
  2022.

\bibitem{abgrall2019reinterpretation}
R.~Abgrall, P.~{\"O}ffner, and H.~Ranocha.
\newblock Reinterpretation and extension of entropy correction terms for
  residual distribution and discontinuous {G}alerkin schemes: Application to
  structure preserving discretization.
\newblock {\em J. Comput. Phys.}, 453:110955, 2022.

\bibitem{abgrall2003high}
R.~{Abgrall} and P.~L. {Roe}.
\newblock {High-order fluctuation schemes on triangular meshes}.
\newblock {\em {J. Sci. Comput.}}, 19(1-3):3--36, 2003.

\bibitem{bacigaluppi2019posteriori}
P.~Bacigaluppi, R.~Abgrall, and S.~Tokareva.
\newblock " a posteriori" limited high order and robust residual distribution
  schemes for transient simulations of fluid flows in gas dynamics.
\newblock {\em arXiv preprint arXiv:1902.07773}, 2019.

\bibitem{zbMATH07524776}
W.~Boscheri, R.~Loub{\`e}re, and P.-H. Maire.
\newblock A 3d cell-centered {ADER} {MOOD} finite volume method for solving
  updated {Lagrangian} hyperelasticity on unstructured grids.
\newblock {\em J. Comput. Phys.}, 449:38, 2022.
\newblock Id/No 110779.

\bibitem{breit2020dissipative}
D.~Breit, E.~Feireisl, and M.~Hofmanov{\'a}.
\newblock Dissipative solutions and semiflow selection for the complete {Euler}
  system.
\newblock {\em Commun. Math. Phys.}, 376(2):1471--1497, 2020.

\bibitem{burman2004edge}
E.~Burman and P.~Hansbo.
\newblock Edge stabilization for {Galerkin} approximations of
  convection-diffusion-reaction problems.
\newblock {\em Comput. Methods Appl. Mech. Eng.}, 193(15-16):1437--1453, 2004.

\bibitem{carrillo2017weak}
J.~A. Carrillo, E.~Feireisl, P.~Gwiazda, and A.~\'{S}wierczewska Gwiazda.
\newblock Weak solutions for {E}uler systems with non-local interactions.
\newblock {\em J. Lond. Math. Soc. (2)}, 95(3):705--724, 2017.

\bibitem{chiodaroli2015global}
E.~Chiodaroli, C.~De~Lellis, and O.~Kreml.
\newblock Global ill-posedness of the isentropic system of gas dynamics.
\newblock {\em Communications on Pure and Applied Mathematics},
  68(7):1157--1190, 2015.

\bibitem{ciarlet2002finite}
P.~G. {Ciarlet}.
\newblock {\em {The finite element methods for elliptic problems.}}, volume~40.
\newblock Philadelphia, PA: SIAM, 2002.

\bibitem{cicchino2022nonlinearly}
A.~Cicchino, S.~Nadarajah, and D.~C. D.~R. Fern{\'a}ndez.
\newblock Nonlinearly stable flux reconstruction high-order methods in split
  form.
\newblock {\em J. Comput. Phys.}, 458:111094, 2022.

\bibitem{clain2011high}
S.~{Clain}, S.~{Diot}, and R.~{Loub\`ere}.
\newblock {A high-order finite volume method for systems of conservation
  laws-multi-dimensional optimal order detection (MOOD)}.
\newblock {\em {J. Comput. Phys.}}, 230(10):4028--4050, 2011.

\bibitem{deLellis2010admissibility}
C.~De~Lellis and L.~Sz\'{e}kelyhidi, Jr.
\newblock On admissibility criteria for weak solutions of the {E}uler
  equations.
\newblock {\em Arch. Ration. Mech. Anal.}, 195(1):225--260, 2010.

\bibitem{deconinck2004residual}
H.~Deconinck and M.~Ricchiuto.
\newblock Residual distribution schemes: foundations and analysis.
\newblock {\em Encyclopedia of computational mechanics}, 2004.

\bibitem{diperna1985compensated}
R.~J. DiPerna.
\newblock Compensated compactness and general systems of conservation laws.
\newblock {\em Trans. Amer. Math. Soc.}, 292(2):383--420, 1985.

\bibitem{feireisl2021note}
E.~Feireisl.
\newblock A note on the long-time behavior of dissipative solutions to the
  {Euler} system.
\newblock {\em J. Evol. Equ.}, 21(3):2807--2814, 2021.

\bibitem{feireisl2019uniqueness}
E.~Feireisl, S.~S. Ghoshal, and A.~Jana.
\newblock On uniqueness of dissipative solutions to the isentropic {E}uler
  system.
\newblock {\em Communications in Partial Differential Equations},
  44(12):1285--1298, 2019.

\bibitem{feireisl2020oscillatory}
E.~Feireisl, C.~Klingenberg, O.~Kreml, and S.~Markfelder.
\newblock On oscillatory solutions to the complete {Euler} system.
\newblock {\em J. Differ. Equations}, 269(2):1521--1543, 2020.

\bibitem{feireisl2019convergence}
E.~Feireisl, M.~Luk\'{a}\v{c}ov\'{a}-Medvid'ov\'{a}, and H.~Mizerov\'{a}.
\newblock Convergence of finite volume schemes for the {E}uler equations via
  dissipative measure-valued solutions.
\newblock {\em Found. Comput. Math.}, 20(4):923--966, 2020.

\bibitem{feireisl2021numerics}
E.~Feireisl, M.~Luk\'{a}\v{c}ov\'{a}-Medvid'ov\'{a}, H.~Mizerov{\'a}, and
  B.~She.
\newblock {\em Numerical analysis of compressible fluid flows}, volume~20 of
  {\em MS\&A, Model. Simul. Appl.}
\newblock Cham: Springer, 2021.

\bibitem{forestier2009}
A.~J. Forestier and P.~Gonzales-Rodelas.
\newblock Implicit schemes of {Lax-Friedrichs} type for systems with source
  terms.
\newblock In {\em Castro Urdiales, Spain}, 7-11. Sep.2009 2009.

\bibitem{friedrich2019entropy}
L.~Friedrich, G.~Schn{\"u}cke, A.~R. Winters, D.~C. D.~R. Fern{\'a}ndez, G.~J.
  Gassner, and M.~H. Carpenter.
\newblock Entropy stable space-time discontinuous {Galerkin} schemes with
  summation-by-parts property for hyperbolic conservation laws.
\newblock {\em J. Sci. Comput.}, 80(1):175--222, 2019.

\bibitem{ghoshal2021uniqueness}
S.~S. Ghoshal and A.~Jana.
\newblock Uniqueness of dissipative solutions to the complete {E}uler system.
\newblock {\em J. Math. Fluid Mech.}, 23(2):Paper No. 34, 25, 2021.

\bibitem{guermond2014second}
J.-L. Guermond, M.~Nazarov, B.~Popov, and I.~Tomas.
\newblock Second-order invariant domain preserving approximation of the {E}uler
  equations using convex limiting.
\newblock {\em SIAM J. Sci. Comput.}, 40(5):A3211--A3239, 2018.

\bibitem{zbMATH07517733}
A.~Haidar, F.~Marche, and F.~Vilar.
\newblock \emph{A posteriori} finite-volume local subcell correction of
  high-order discontinuous {Galerkin} schemes for the nonlinear shallow-water
  equations.
\newblock {\em J. Comput. Phys.}, 452:34, 2022.
\newblock Id/No 110902.

\bibitem{harten1983symmetric}
A.~Harten.
\newblock On the symmetric form of systems of conservation laws with entropy.
\newblock {\em J. Comput. Phys.}, 49(1):151--164, 1983.

\bibitem{hubbard2011discontinuous}
M.~E. {Hubbard} and M.~{Ricchiuto}.
\newblock {Discontinuous upwind residual distribution: a route to unconditional
  positivity and high order accuracy}.
\newblock {\em {Comput. Fluids}}, 46(1):263--269, 2011.

\bibitem{huynh2007flux}
H.~Huynh.
\newblock A flux reconstruction approach to high-order schemes including
  discontinuous {G}alerkin methods.
\newblock {\em AIAA paper}, 4079:2007, 2007.

\bibitem{zbMATH07446443}
B.~Jin, Y.-S. Kwon, {\v{S}}.~Ne{\v{c}}asov{\'a}, and A.~Novotn{\'y}.
\newblock Existence and stability of dissipative turbulent solutions to a
  simple bi-fluid model of compressible fluids.
\newblock {\em J. Elliptic Parabol. Equ.}, 7(2):537--570, 2021.

\bibitem{kuzmin2020monolithic}
D.~Kuzmin.
\newblock Monolithic convex limiting for continuous finite element
  discretizations of hyperbolic conservation laws.
\newblock {\em Comput. Methods Appl. Mech. Engrg.}, 361:112804, 28, 2020.

\bibitem{kuzmin2022limiter}
D.~Kuzmin, H.~Hajduk, and A.~Rupp.
\newblock Limiter-based entropy stabilization of semi-discrete and fully
  discrete schemes for nonlinear hyperbolic problems.
\newblock {\em Comput. Methods Appl. Mech. Eng.}, 389:28, 2022.
\newblock Id/No 114428.

\bibitem{zbMATH07437238}
Y.-S. Kwon and A.~Novotn{\'y}.
\newblock Construction of weak solutions to compressible {Navier}-{Stokes}
  equations with general inflow/outflow boundary conditions via a numerical
  approximation.
\newblock {\em Numer. Math.}, 149(4):717--778, 2021.

\bibitem{zbMATH07559963}
Y.~Li and B.~She.
\newblock A numerical approach for the existence of dissipative weak solutions
  to a compressible two-fluid model.
\newblock {\em J. Math. Fluid Mech.}, 24(3):17, 2022.
\newblock Id/No 78.

\bibitem{bangwei2022convergence}
Y.~Li and B.~She.
\newblock On convergence of numerical solutions for the compressible {MHD}
  system with exactly divergence-free magnetic field.
\newblock {\em SIAM J. Numer. Anal.}, 60(4):2182--2202, 2022.

\bibitem{lukacova2022convergence}
M.~Luk{\'a}{\v{c}}ov{\'a}-Medvid'ov{\'a} and P.~{\"O}ffner.
\newblock Convergence of discontinuous {Galerkin} schemes for the {Euler}
  equations via dissipative weak solutions.
\newblock {\em Appl. Math. Comput.}, 436:22, 2023.
\newblock Id/No 127508.

\bibitem{lukavcova2021convergence}
M.~Luk\'{a}\v{c}ov\'{a}-Medvid'ov\'{a} and Y.~Yuan.
\newblock Convergence of first-order finite volume method based on exact
  {R}iemann solver for the complete compressible {E}uler equations.
\newblock {\em arXiv preprint arXiv:2105.02165}, 2021.

\bibitem{michel2021spectral}
S.~{Michel}, D.~{Torlo}, M.~{Ricchiuto}, and R.~{Abgrall}.
\newblock {Spectral analysis of continuous FEM for hyperbolic PDEs: influence
  of approximation, stabilization, and time-stepping}.
\newblock {\em {J. Sci. Comput.}}, 89(2):41, 2021.
\newblock Id/No 31.

\bibitem{zbMATH07568430}
V.~Michel-Dansac and A.~Thomann.
\newblock {TVD}-{MOOD} schemes based on implicit-explicit time integration.
\newblock {\em Appl. Math. Comput.}, 433:23, 2022.
\newblock Id/No 127397.

\bibitem{zbMATH07229478}
H.~Mizerov{\'a} and B.~She.
\newblock Convergence and error estimates for a finite difference scheme for
  the multi-dimensional compressible {Navier}-{Stokes} system.
\newblock {\em J. Sci. Comput.}, 84(1):39, 2020.
\newblock Id/No 25.

\bibitem{ni1981multiple}
R.-H. Ni.
\newblock A multiple grid scheme for solving the {E}uler equations.
\newblock In {\em 5th Computational Fluid Dynamics Conference}, page 1025,
  1981.

\bibitem{offner2020stability}
P.~{\"O}ffner.
\newblock {\em Approximation and {S}tability {P}roperties of {N}umerical
  {M}ethods for {H}yperbolic {C}onservation {L}aws}.
\newblock {H}abilitation, University Zurich (to appear Springer 2022), 2020.

\bibitem{PerthameShu1996positivity}
B.~{Perthame} and C.-W. {Shu}.
\newblock {On positivity preserving finite volume schemes for Euler equations}.
\newblock {\em {Numer. Math.}}, 73(1):119--130, 1996.

\bibitem{ricchiuto2010explicit}
M.~{Ricchiuto} and R.~{Abgrall}.
\newblock {Explicit Runge-Kutta residual distribution schemes for time
  dependent problems: second order case}.
\newblock {\em {J. Comput. Phys.}}, 229(16):5653--5691, 2010.

\bibitem{roe1981approximate}
P.~L. Roe.
\newblock Approximate {Riemann} solvers, parameter vectors, and difference
  schemes.
\newblock {\em J. Comput. Phys.}, 43:357--372, 1981.

\bibitem{roe1982fluctuations}
P.~L. Roe.
\newblock Fluctuations and signals - {A} framework for numerical evolution
  problems.
\newblock Numerical methods for fluid dynamics, {Proc}. {Conf}.,
  {Reading}/{U}.{K}. 1982, 219-257 (1982)., 1982.

\bibitem{roe1986characteristic}
P.~L. Roe.
\newblock Characteristic-based schemes for the {E}uler equations.
\newblock {\em Annual review of fluid mechanics}, 18(1):337--365, 1986.

\bibitem{trojak2022extended_2}
W.~Trojak and P.~Vincent.
\newblock An extended range of energy stable flux reconstruction methods on
  triangles.
\newblock {\em arXiv preprint arXiv:2203.15848}, 2022.

\bibitem{trojak2022extended}
W.~Trojak, R.~Watson, and P.~Vincent.
\newblock An extended range of stable flux reconstruction schemes on
  quadrilaterals for various polynomial bases.
\newblock {\em arXiv preprint arXiv:2206.01015}, 2022.

\bibitem{zbMATH07300580}
F.~Vilar.
\newblock \emph{A posteriori} correction of high-order discontinuous {Galerkin}
  scheme through subcell finite volume formulation and flux reconstruction.
\newblock {\em J. Comput. Phys.}, 387:245--279, 2019.

\bibitem{WuShu}
K.~Wu and C.~Shu.
\newblock Geometric quasilinearization framework for analysis and design of
  bound-preserving schemes.
\newblock {\em SIREV}, submitted.
\newblock \url{https://arxiv.org/abs/2111.04722}.

\end{thebibliography}

\end{document}